\documentclass[11pt, leqno, twoside]{article}

\usepackage{amssymb}
\usepackage{amsmath}
\usepackage{amsthm}
\usepackage{amsfonts}
\usepackage{mathrsfs}
\usepackage{indentfirst}
\usepackage{graphicx}

\usepackage{txfonts}

\usepackage{anysize}

\usepackage{fancyhdr}

\usepackage{color}
\usepackage{setspace}

\usepackage{amsmath}

\usepackage{enumerate}
\usepackage{graphicx}
\usepackage{caption}
\usepackage{color}
\usepackage{amssymb}
\newtheorem{theorem}{Theorem}[section]
\newtheorem{lemma}[theorem]{Lemma}
\newtheorem{proposition}[theorem]{Proposition}
\newtheorem{corollary}[theorem]{Corollary}

\theoremstyle{remark}
\newtheorem{remark}{Remark}[section]

\numberwithin{equation}{section}

 \DeclareMathOperator\hdim{\dim_H}

\DeclareMathOperator\loc{\rm loc}

 \DeclareMathOperator\supp{\rm supp}
 
 \DeclareMathOperator{\d_H}{d_{\rm H}}

\def\N{\mathbb{N}}
\def\Z{\mathbb{Z}}
\def\R{\mathbb{R}}

\allowdisplaybreaks

\usepackage{txfonts}

\def\wz{\widetilde}

\def\supp{\mathop\mathrm{\,supp\,}}

\def\loc{{\mathop\mathrm{\,loc\,}}}

\def\q1{\wz q}
\def\Q1{q_1}

\def\loc{{\mathop\mathrm{loc}}}

\newtheorem{thm}{Theorem}[section]

\theoremstyle{definition}

\newtheorem{rem}[thm]{Remark}

\newtheorem{Example}[thm]{Example}

\numberwithin{equation}{section}

\begin{document}
\title{\bf\Large Quantitative Hardy--Littlewood maximal inequalities and Wiener--Stein theorem on p.c.f. self-similar fractals
\footnotetext{\hspace{-0.35cm}
2020 {\it Mathematics Subject Classification}.
{Primary 42B25; Secondary 28A12, 28A25, 31C15.} \endgraf
{\it Key words and phrases.} Self-similar set; Hausdorff content; Hardy--Littlewood maximal operator; Choquet integral;  $L\log L$. \endgraf}}
\author{Long Huang, Jinjun Li\footnote{Corresponding author, E-mail: lijinjun@gzhu.edu.cn
/{\color{red}{January 13, 2026}}}\,\,  and Xiaofeng Wang
}
\date{ }
\maketitle
	
\vspace{-0.7cm}

\begin{center}
\begin{minipage}{13cm}
{\small {\bf Abstract}\quad
Let $K\subset \mathbb{R}^d$ be a post-critically finite (p.c.f.) self-similar  set with Hausdorff dimension $s$, and $\mu$ be a self-similar probability measure supported on $K$. Let $H^{\alpha}_\mu$, $0<\alpha\le s$, be the Hausdorff content  on $K$, and $M_{\mathcal{D}}^\mu $ be the Hardy--Littlewood maximal operator on $K$ associated with its basic cubes $\mathcal{D}$. In this paper, we establish quantitative strong type and weak type Hardy--Littlewood maximal inequalities on $K$ with respect to Hausdorff content  $H^{\alpha}_\mu$ for all $0<\alpha\le s$.
As applications, via the maximal operator $M_{\mathcal{D}}^\mu $, we characterize the Lebesgue--Choquet space $L^p(K,H^{\alpha}_\mu)$ and the Zygmund space $L\log L(K,\mu)$.
To be exact, given $\alpha/s< p\le \infty$, we discover that
\begin{align*}
\text{$f\in L^p(K,H^{\alpha}_\mu)$ if and only if $M_{\mathcal{D}}^\mu f\in L^p(K,H^{\alpha}_\mu)$}
\end{align*}
 and, for $f\in L^1(K,\mu)$ with $K$ satisfying the strong separation condition,
\[
\text{$M_{\mathcal{D}}^\mu f\in L^1(K,\mu)$  if and only if $f\in L\log L(K,\mu)$}.
\]
That is, Wiener's $L\log L$ inequality and its converse
inequality, due to Stein in 1969, are established on fractal sets.
}
\end{minipage}
\end{center}

\vspace{0.2cm}
\maketitle
\begin{center}
	\tableofcontents
\end{center}
\section{Introduction}

\overfullrule=5pt
Let $L_{\loc}^1({{\mathbb R}^d})$ denote
the set of all locally Lebesgue integrable functions on
${{\mathbb R}^d}$. The \emph{Hardy--Littlewood maximal operator}
is defined by setting, for any $f\in L_{\loc}^1({{\mathbb R}^d})$ and $x\in{{\mathbb R}^d}$,
$${\mathcal M}f(x):=\sup_{x\in B}\frac1{\mathcal{L}^d(B)}\int_{B}|f(y)|\,dy,$$
where the supremum is taken over all balls $B$ containing $x$ and $\mathcal{L}^d$  denotes the Lebesgue measure on ${\mathbb R}^d$. Equivalently, the supremum may be taken over all balls centered at $x$, or over all cubes $Q$ containing $x$; see, for example, \cite{swbook}.

As an important and fundamental result in harmonic analysis,
the Hardy--Littlewood maximal operator $\mathcal{M}$ is strong-type $(p,p)$ for $p\in(1, \infty)$, that is, for any given $p\in(1,\infty)$,
there exists a positive constant $C$, depending only on
$d$ and $p$, such that, for any  $f\in L^p(\mathbb{R}^d)$,
\[
     \left[\int_{{{\mathbb R}^d}}[\mathcal{M} f(x)]^p\,dx\right]^{\frac1p} \leq
 C \left[\int_{{{\mathbb R}^d}}|f(x)|^p\,dx\right]^{\frac1p}.
\]
But for $f\in L^1(\mathbb{R}^d)$ the last inequality is not true, i.e., the maximal operator,
as a mapping defined on $L^1(\mathbb{R}^d)$ is not a bounded transformation into $ L^1(\mathbb{R}^d)$. Indeed, for any $f\in L^1(\mathbb{R}^d)$ with $f\neq 0$, then $\mathcal{M} f\notin L^1(\mathbb{R}^d)$.
In this case, the operator $\mathcal{M}$, however, is weak-type $(1,1)$, that is, there exists a positive constant $C$, depending only on $d$, such that, for any $f \in L^1(\mathbb{R}^d)$,
   \[
        \sup_{t\in(0,{\infty})}t \mathcal{L}^d \left(\{ x\in \mathbb{R}^d: \mathcal{M} f(x) >t  \}   \right)
\le C \int_{\mathbb{R}^d}  |f(x)| \; d x.
   \]
These celebrated inequalities were proved by Hardy and Littlewood \cite{hl30} for $d=1$ and by Wiener \cite{w39} for $d\ge 2$. In addition, for $f\in L^1(\mathbb{R}^d)$ with $f\neq 0$, $\mathcal{M}f$ even does not need to be locally integrable.

However, in the seminal paper \cite{w39}, Wiener showed that
$\mathcal{M} f$ is locally integrable if $f$ is in the Zygmund space $L\log L$, i.e.,
$$\int_{\R^d}|f(x)|\log(e+|f(x)|)\,dx<\infty,$$
where $e$ is the Euler number.
More precisely, there exists a positive constant $C$ such that for any given ball $B$,
\begin{align}\label{25eq1}
\int_B\mathcal{M} f(x)\,dx\leq 2\mathcal{L}^d(B)+C\int_{\R^d}|f(x)|\log(e+|f(x)|)\,dx.
\end{align}
In 1969, Stein \cite{s69}  proved that the converse of
Wiener's $L\log L$ inequality \eqref{25eq1} is also true, that is,
if $f$ is integrable and is supported on some finite ball $B$ and $\int_B\mathcal{M} f(x)\,dx<\infty$, then there exists a positive constant $C$ such that
$$\int_{B}|f(x)|\log(e+|f(x)|)\,dx\leq C\int_B\mathcal{M} f(x)\,dx.$$


On another hand, fractals are fascinating
mathematical objects that exhibit irregularity and self-similarity at different scales. Thus, fractals can be employed to approach the true characteristics and states for complex systems. Therefore there has been growing interest in developing analysis on fractals; see, for example, a series of significant works by
Strichartz \cite{Str2,Str3,Str4}.

In this paper, we follow the line of Strichartz and develop analysis related to Hardy--Littlewood maximal operator on p.c.f. self-similar fractal set $K\subset \R^d$. To be specific, this paper
aims to establish quantitative strong type and weak type Hardy--Littlewood maximal inequalities on the p.c.f. self-similar fractal set $K\subset \R^d$ equipped with Hausdorff contents $H^{\alpha}_\mu$. As applications, we characterize Lebesgue--Choquet spaces and Zygmund spaces on given self-similar set $K\subset \R^d$, that is,
Wiener's $L\log L$ inequality and its converse
inequality due to Stein in 1969  for Hardy--Littlewood maximal operator on  $K$ are also obtained.
These results, obtained on this paper, have a wide range of applications, since the symmetric Cantor set, middle-forth Cantor sets, generalized Sierpi\'{n}ski carpet, Sierpi\'{n}ski  gasket and Vicsek set  and so on are examples of the set $K$ under consideration. More details on these examples will be given in Section 2.1.

We shall briefly sketch the main difficulties in this paper:\vspace{-0.1cm}
\begin{enumerate}[(1)]
\item [\textup{(a)}] dyadic decomposition technique,
the most fundamental tool for proving the classical maximal inequalities in Euclidean spaces, cannot be applied to p.c.f. self-similar fractals.
\vspace{-0.1cm}
\item [\textup{(b)}] these quantitative inequalities are built under the Hausdorff content setting, which provides a finer measurement of set size compared to the Lebesgue measure and is just a outer measure rather than a measure; moreover, it coincides with Lebesgue measure as a special case.
\vspace{-0.1cm}
\item [\textup{(c)}] the Choquet integral under consideration is not linear, and it is well known that the linearity of integrals plays a significant and fundamental role, as demonstrated in the proof of the classical Hardy--Littlewood maximal inequalities.
\end{enumerate}\vspace{-0.1cm}

Thus, to prove these desired results, we combine the methods and techniques coming both from harmonic analysis and fractal geometry. For example, to prove the main results, we combined the idea of dyadic decomposition in harmonic analysis with the basic construction cubes in fractal geometry to overcome the above difficulty (a) and, by proving crucial Lemma \ref{zfg}, we tackle the difficulty (c) of lacking linearity.
This paper may shed some light on developing harmonic analysis on fractal set, since Hardy--Littlewood maximal operator is one of the most fundamental operators in harmonic analysis.

To state the main results, we need to introduce some notions.
For $1\le i\le m,$ let $S_i:\R^d\to \R^d$ be contractive similitudes,
\begin{align*}
S_i(x)=r_iR_i(x)+b_i,
\end{align*}
where $0<r_i<1,\ b_i\in \R^d$ and each $R_i$ is an orthogonal transformation. We call $\{S_i\}_{i=1}^m$ an \emph{iterated function system (IFS)}. It is well known that there is a unique non-empty compact set $K\subset \R^d$ such that $K=\bigcup_{i=1}^mS_i(K)$; see \cite{Hutch}. The set $K$ is called the \textit{self-similar set} generated by $\{S_i\}_{i=1}^m$. In this paper, we always assume that $K$ is non-trivial, by which we mean that the contractions do not all share a common fixed point, and thus $K$ is always uncountable. Further, given a probability vector $\textbf{p}=(p_1,\ldots,p_m)$, i.e., $p_i>0$ for $1\le i\le m$ and $\sum_{i=1}^m p_i=1,$ there is a unique Borel probability measure $\mu=\mu_{\textbf{p}}$ on $\R^d$ satisfying
\begin{equation*}
\mu=\sum_{i=1}^m p_i\mu\circ S_i^{-1}.
 \end{equation*}
The measure $\mu$ is supported on $K$ and is called a \textit{self-similar measure}. It is well-known that $\mu$ is continuous under our assumption.

 For any $n\ge 1$, let $\Sigma_n:=\{1,\ldots, m\}^n$ be the set of  all finite sequences
$\textbf{i}=i_{1}\ldots i_{n}$ of length $n$ with
$i_{j}\in\{1,\ldots,m\}$. We define $\Sigma_0=\emptyset$ by convention. Let
\[
\Sigma^*:=\bigcup_{n\ge 1}\Sigma_n
\]
denote the set of all finite sequences. For $\textbf{i}=i_{1}\ldots i_{n}\in \Sigma_n$, define $S_{\textbf{i}}=S_{i_1}\circ S_{i_2}\circ \ldots \circ S_{i_n}$ and $K_{\textbf{i}}=S_{\textbf{i}}(K)$. Of course, let $K_\emptyset=K.$ Write
\[
\mathcal{D}:=\bigcup_{n=0}^\infty\{K_{\textbf{i}}:\textbf{i}\in \Sigma_n \}.
 \]
The elements in $\mathcal{D}$ are called the \emph{basic cubes} of $K.$  For any $x\in K$ and integer $n\ge 1$, there exists a (not necessarily unique)  $\textbf{i}\in \Sigma_n$ such that $K_{\textbf{i}}$ contains $x$, and this basic cube is denoted by $K_{\textbf{i}}(x)$; see Section 2.1. Sometimes we replace  $K_{\textbf{i}}(x)$ by $K_{\textbf{i}_n}(x)$ to  emphasize the length of $\textbf{i}$.

For $f\in L^1(K,\mu),$ the \emph{Hardy--Littlewood maximal operator} $M_{\mathcal{D}}^\mu$ associated with $\mu$ is defined by, for any $x\in K$,
\[
M_{\mathcal{D}}^\mu f(x):=\sup_{n\in \N}\frac{1}{\mu(K_{\textbf{i}_n}(x))}\int_{K_{\textbf{i}_n}(x)}|f(y)|d\mu(y).
\]
Throughout the paper, $s$ always denotes the Hausdorff dimension of $K$ and a nonnegative function always refers to a $\mu$-measurable nonnegative function.
 For $E\subset K$ and $0<\alpha\le s$, the \emph{$\alpha$-dimensional Hausdorff content} of $E$ associated with $\mu$ is defined by
\begin{align}\label{hc}
H^\alpha_\mu(E):=\inf\left\{\sum_{j}\mu(I_j)^{\alpha/s}:\ E\subset \bigcup_{j}I_j,\ I_j\in \mathcal{D}\right\}.
\end{align}
As usual, the \emph{Choquet integral} of a nonnegative function $f$ with respect to $H^\alpha_\mu$ is defined by
\[
\int_K f(x)d H^\alpha_\mu:=\int_0^\infty H^\alpha_\mu(\{x\in K: f(x)>t\})dt.
\]
For  $p\in(0,{\infty})$, the \emph{$p$-Choquet
integral of function $f$ with respect to  $H^\alpha_\mu$}
is defined by
\begin{align}\label{lp}
\begin{split}
\|f\|_{L^p(K,H^\alpha_\mu)}^p&=\int_{K}|f(x)|^p\,dH^\alpha_\mu\\
&:=p\int_{0}^{\infty}{t}^{p-1}
H^\alpha_\mu\left(\left\{x\in K:\ |f(x)|>t\right\}\right)
dt.
\end{split}
\end{align}
For more details on Hausdorff contents and Choquet integrals, we refer to Adams \cite{a98} and Tang \cite{t11}.

In what follows, for any $p\in(0,{\infty})$, we always denote
by $L^p(K,H^\alpha_\mu)$ the \emph{Lebesgue--Choquet space} of all functions
$f$ on $K$ such that $\|f\|_{L^p(K,H^\alpha_\mu)}^p$ as in \eqref{lp}
are finite. We point out that, unlike Riemann or Lebesgue integrals,
the Choquet integral is a nonlinear integral and well defined for all nonnegative functions, without an assumption of measurability. But in  what follows we will only restrict ourselves to measurable functions. Furthermore, when limiting functions to measurable on Borel set of $K$, then the Lebesgue--Choquet space $L^p(K,H^s_\mu)$ coincides with the Lebesgue space $L^p(K,\mu)$ by
Lemma \ref{eq} below.

Now we can state our main results.
In what follows, $K$ and $\mu$ always, respectively, denote the self-similar set generated by $\{S_i\}_{i=1}^m$  and the self-similar measure. Let $
H^\alpha_\mu(E)$, $0<\alpha\le s$, denote
the $\alpha$-dimensional Hausdorff content of $E\subset K$ defined by
\eqref{hc} in the rest of the paper.

Our first result is the following quantitative strong type $(p,p)$ inequality under $\alpha$-dimensional Hausdorff contents. The notion of p.c.f. self-similar set and the strong separation condition will be given in Section 2.1.
\begin{theorem}\label{MR1}
Assume that $K$ is a p.c.f. self-similar set. Let $0<\alpha< s$ and $\frac{\alpha}{s}<p<\infty$. Then, for any $f\in L^p(K,H^\alpha_\mu)$,
\[
\int_K [M_{\mathcal{D}}^\mu f(x)]^p dH^\alpha_\mu\le C_{p,\alpha,s}\int_K |f(x)|^p dH^\alpha_\mu,
\]
where
$$
 C_{p,\alpha,s}:=\begin{cases}
\frac{2^{p+2}}{p-\alpha/s} \ \ &\text{when}\ \ \alpha/s<p<1\\
 \frac{2^{2p+1}}{p(1-\alpha/s)} \ \ &\text{when}\ \ 1\le p<\infty.
\end{cases}
$$
\end{theorem}

At the endpoint case, we have the quantitative weak type $(\alpha/s,\alpha/s)$ inequality with Hausdorff contents:
\begin{theorem}\label{MR2}
Assume that $K$ is a p.c.f. self-similar set. Let $0<\alpha\le s$. Then, for any $f\in L^{\frac \alpha s}(K,H^\alpha_\mu)$ and $t>0$,
\[
H^\alpha_\mu(\{x\in K: M_{\mathcal{D}}^\mu f(x)>t\})\le 4\left(\frac{s}{\alpha}\right)^{\frac{\alpha}{s}}t^{-\frac{\alpha}{s}}\int_K |f(x)|^{\frac{\alpha}{s}}dH^\alpha_\mu.
\]
\end{theorem}

\begin{rem}
Here we would like emphasis that the measure $\mu$ on a p.c.f. self-similar set $K$ is not necessarily doubling and hence the corresponding Hausdorff content in Theorems \ref{MR1} and \ref{MR2} is not ``doubling". The above two theorems provide Hardy--Littlewood maximal inequalities on ``non-doubling" Hausdorff content. Recall that if the measure is doubling, then the Hardy--Littlewood maximal operator is of weak type
$(1,1)$ and hence of strong type $(p,p)$ for any $p\in (1,\infty)$. However, in general, weak type $(1,1)$ inequalities may not hold. Moreover, it is even possible to construct metric measure spaces for which the associated maximal operator fails to be of weak type  $(p,p)$ for every $p\in [1,\infty)$; see, for example, the works of Aldaz \cite{a12} and Li \cite{Li04,Li05,Li07}.
\end{rem}

\begin{rem}
The novel phenomenon under Hausdorff contents is that the endpoint $\alpha/s$ allows less than 1. Thus, the range $(\alpha/s,\infty)$ of $p$ in Theorem \ref{MR1} is larger than the usual $(1,\infty)$ due to the Hausdorff content. Moreover, under the Hausdorff content setting, the integral is not linear, and hence the weak type inequality is hard to prove, where the classical methods on Lebesgue spaces fail. Thus, in the paper, we directly show  Theorem \ref{MR1} by passing from characteristic functions to any $f\in L^p(K,H^\alpha_\mu)$ with the help of Lemma \ref{tzfgj}. We point out that one can also establish Theorem \ref{MR1} from the harder weak type inequality (Theorem \ref{MR2}) by the standard argument as the classical case.
\end{rem}

\begin{rem}
Determining the optimal constants in Theorems \ref{MR1} and \ref{MR2} would be a challenging and interesting topic. For the optimal constants in the classical Hardy--Littlewood maximal inequalities, we refer to the celebrated work by Melas \cite{m03,m02} and Aldaz \cite{a11}.
\end{rem}

As a consequence of Theorem \ref{MR2} with $s=\alpha$, we have the following weak type $(1,1)$ inequality.

\begin{corollary}\label{weaki}
Assume that $K$ is a p.c.f. self-similar set. For $f\in L^1(K,\mu)$ and any $t>0$, we have
\[
\mu\left(\left\{x\in K: M_{\mathcal{D}}^\mu f(x)>t\right\}\right)\le \frac{4}{t}\int_K |f(x)|d\mu.
\]
\end{corollary}
We point out that, it is well-known that the above weak-type inequality in Corollary \ref{weaki} holds for doubling measure, see, for example \cite{Str}. In particular,  Fu et al. \cite{FTW} obtained a similar inequality  for a special class of Bernoulli convolutions, and used it to study the mock Fourier series on fractal sets. The difference is that the measure $\mu$ here is not necessarily doubling and we obtain a better upper bound.

Applying Corollary \ref{weaki},
we have the following strong type $(p,p)$ inequality.

\begin{corollary}\label{pp}
Assume that $K$ is a p.c.f. self-similar set. For $1<p<\infty$ and $f\in L^p(K,\mu)$, we have
\[
\int_K [M_{\mathcal{D}}^\mu f(x)]^p d\mu\le
\frac{2^{p+2}p}{p-1}\int_K |f(x)|^p d\mu.
\]
\end{corollary}

As an application of Corollary \ref{weaki}, we show the Lebesgue differentiation theorem on the self-similar set $K$ (see Theorem \ref{A2}). This Lebesgue differentiation theorem, combined with Theorem \ref{MR1}, further induces the following equivalent characterization of Lebesgue--Choquet spaces $L^p(K,H^{\alpha}_\mu)$.
In what follows, let $L^{\infty}(K,H^{\alpha}_\mu)$ denote the space of all functions $g$ on $K$ such that the quasi-norm
$$\|g\|_{L^\infty(K,H^{\alpha}_\mu)}:=\inf_{E_0\subset K,\,H^{\alpha}_\mu(E_0)=0}\sup_{x\in K\backslash E_0}|g(x)|<\infty.$$

\begin{theorem}\label{lpc}
Assume that $K$ is a p.c.f. self-similar set. Let $0<\alpha\le s$ and $\alpha/s< p\le \infty$. Then
$f\in L^p(K,H^{\alpha}_\mu)$ if and only if $M_{\mathcal{D}}^\mu f\in L^p(K,H^{\alpha}_\mu)$. Moreover, there exists a positive constant $C$ such that
\[
\|f\|_{ L^p(K,H^{\alpha}_\mu)}\le \| M_{\mathcal{D}}^\mu f\|_{ L^p(K,H^{\alpha}_\mu)}\le C\|f\|_{ L^p(K,H^{\alpha}_\mu)}.
\]
\end{theorem}

From function spaces viewpoint, Corollary \ref{pp} yields the bounded mapping property:
\[M_{\mathcal{D}}^\mu:\ L^p(K,\mu)\to L^p(K,\mu)
\]
for any $1<p<\infty$, and in the endpoint case $p=1$, Corollary \ref{weaki} yields the bounded mapping property:
\[M_{\mathcal{D}}^\mu:\ L^1(K,\mu)\to L^{1,\infty}(K,\mu).\]
Here $L^{1,\infty}(K,\mu):=\{f:\ \sup_{t>0}t\mu\{x\in K: |f(x)|>t\}<\infty\}$.
Then it is naturally to ask: Which function
spaces $X$ such that the maximal operator $M^\mu_{\mathcal{D}}$ is bounded from $X$ to $L^1(K,\mu)$, i.e., for which $X$,
\[M_{\mathcal{D}}^\mu:\ X\to L^{1}(K,\mu)\]
is a bounded mapping?
 When we are considering this problem, the Zygmund space
$L\log L(K,\mu)$ enters the picture. Namely, we establish a similar  Wiener's $L\log L$ inequality on p.c.f. self-similar set $K$ as follows to answer this question.

\begin{proposition}\label{MR3}
Assume that $K$ is a p.c.f. self-similar set.  For  $|f|\log^+|f|\in L^1(K,\mu)$, we have
\[
\int_K M_{\mathcal{D}}^\mu f(x) d\mu\le 2\mu(K)+8\int_K |f(x)|\log^+|f(x)| d\mu,
\]
where $\log^+t:=\max\{0,\log t\}$ for any $t>0$.
\end{proposition}

Define $L\log L(K,\mu)$ to be the family of all $\mu$-measurable
function $f$ on $K$ for which
\[
\int_K |f(x)|\log^+|f(x)| d\mu<\infty.
\]
Proposition \ref{MR3} shows that $M^\mu_{\mathcal{D}}$ is a bounded operator from $L\log L(K,\mu)$ to $L^1(K,\mu)$, that is, it yields the following bounded mapping property:
\[
M^\mu_{\mathcal{D}}:\ L\log L(K,\mu)\to L^1(K,\mu).
\]
We remark that Proposition \ref{MR3} is also true for
$p=\frac{\alpha}{s}$ with $0<\alpha<s$ and
replacing $\mu$ by the Hausdorff content $H^\alpha_\mu$, that is, the operator $M^\mu_{\mathcal{D}}$ is bounded from $L^{\frac{\alpha}{s}}\log L(K,H^\alpha_\mu)$ to $L^\frac{\alpha}{s}(K,H^\alpha_\mu)$. Here,
$L^{\frac \alpha s}\log L(K,H^\alpha_\mu)$ denotes the family of all $H^\alpha_\mu$-measurable
function $f$ on $K$ for which
\[
\int_K |f(x)|^{\frac \alpha s}\log^+|f(x)| dH^\alpha_\mu<\infty.
\]

Under a stronger condition, the converse of Proposition \ref{MR3} is also proved, which is a counterpart of Stein's $L\log L$ inequality from \cite{s69} in 1969.  Recall that we say that IFS $\{S_i\}_{i=1}^m$ satisfies the \emph{strong separation condition} (SSC) if $S_i(K)\cap S_j(K)=\emptyset$ for all $i\not=j$. In this case, $K$ is called a self-similar set with SSC. Note that, under the SSC, the self-similar measure $\mu$ is doubling, i.e., there exists a positive constant $D$ such that $\mu(B_{2r}(x))\le D\mu(B_r(x))$ for all $x\in K$ and $r>0$; see Olsen \cite{Olsen}. Here and in the sequel, $B_r(x)$ always denotes the ball with the center $x$ and radius $r.$

\begin{proposition}\label{t16}
Assume that $K$ is a self-similar set with SSC and  $f\in L^1(K,\mu)$ and $M_{\mathcal{D}}^\mu f\in L^1(K,\mu)$. Then $|f|\log^+|f|\in L^1(K,\mu)$ and there exists a positive constant $C$ such that
\[
\int_K |f(x)|\log^+|f(x)| d\mu(x)\le C\int_K M_{\mathcal{D}}^\mu f(x) d\mu(x).
\]
\end{proposition}

Combining Propositions \ref{MR3} and \ref{t16}, we further conclude the following Wiener--Stein theorem, which provides an equivalent characterization of Zygmund spaces $L\log L(K,\mu)$.

\begin{theorem}
Suppose  $K$ is a self-similar set with SSC and $f\in L^1(K,\mu)$. Then $M_{\mathcal{D}}^\mu f\in L^1(K,\mu)$ if and only if $f\in L\log L(K,\mu)$.
\end{theorem}

Finally, let us outline the structure of the paper. In Section \ref{s2}, we recall some basic facts on self-similar measures and prove some properties on Hausdorff contents and Choquet integrals, which are needed in our proofs. Section \ref{s3} is devoted to the proofs of Theorems \ref{MR1} and \ref{MR2},
and Corollaries \ref{weaki} and \ref{pp}. Finally, the proofs of Theorem \ref{lpc}, Propositions \ref{MR3} and \ref{t16} are given in Section \ref{s4}.

\section{Preliminaries}\label{s2}
\subsection{Self-similar sets and measures}In this subsection, we collect some basic properties of self-similar sets and measures that are needed in this paper. Many fractals are made up of parts that are, in some way, similar to the whole. For example, the middle third Cantor set is the union of two similar copies of itself. These fractals are called self-similar sets and can be obtained by IFSs. The following lemma, due to Hutchinson \cite{Hutch}, guarantees the existence of self-similar set $K$ and self-similar measure $\mu$ on it.

\begin{lemma}
Let $\{S_i\}_{i=1}^m$ be contractive similitudes with contraction factors $r_1, \ldots, r_m$ on $\R^d$. Then
\begin{enumerate}[(1)]
 \item [\textup{(i)}] There exists a unique nonempty compact set $K\subset \R^d$  which satisfies
 \begin{equation}\label{sss}
 K=\bigcup_{i=1}^m S_i(K).
 \end{equation}

\item [\textup{(ii)}]  For any probability vector $\textbf{p}=(p_1,\ldots, p_m)$, there is a unique probability measure $\mu=\mu_\textbf{p}$  on $K$ such that
    \[
    \mu=\sum_{i=1}^m p_i \mu\circ S_i^{-1}.
    \]
 If $p_i>0$ for all $1\le i\le m$, then $\supp \mu =K.$
 \end{enumerate}
\end{lemma}


Iterating invariant equation \eqref{sss} we obtain
\[
K=\bigcup_{n=1}^\infty\bigcup_{i_1i_2\ldots i_n\in \Sigma_n}S_{i_1}\circ S_{i_{2}}\circ\cdots \circ S_{i_n}(K)=\bigcup_{n=1}^\infty\bigcup_{\textbf{i}\in \Sigma_n}K_{\textbf{i}}.
\]
Suppose that $\{S_1, \ldots, S_m\}$ is an similar IFS on $\R^d$ and $K$ is the corresponding self-similar set. Let $\Sigma_\infty:=\{1,\ldots, m\}^\N$ be the set of all infinite sequences with values in $\{1,\ldots, m\}$. For $\mathbf{i}=i_{1}i_{2}\ldots \in \Sigma_\infty$ and a
positive integer $n$, let $\mathbf{i}|_n=i_{1}\ldots i_{n}$
denote the truncation of $\mathbf{i}$ to the $n$th place.

It is well known (see \cite{Fal,Hutch}) that there exists a continuous and surjective mapping $\Pi: \Sigma_\infty\to K$ satisfying
\begin{equation*}
\Pi (\textbf{i})=\bigcap_{n=1}^\infty K_{\textbf{i}|_n}.
\end{equation*}
Recall that $K_\textbf{i}=S_\textbf{i}(K)$.
If additionally, all $S_1,\ldots,S_m $ are injective and $S_i(K)\cap S_j(K)=\emptyset$ for all $i\not= j$, then $\Pi$ is a homeomorphism.

We give some comments and geometric intuition on the self-similar set $K$ satisfying \eqref{sss}. Let $\mathcal{S}(\R^d)$ denote the collection of nonempty compact
subsets of $\R^d$. For $\delta>0$ and $A\subset \R^d$, the $\delta$-neighbourhood of $A$ is the set of points within distance $\delta$ of $A$, i.e.,
\[
A_\delta:=\{x\in X: \text{$d(x, a)< \delta$ for some $a\in A$}\}.
\]
 For two sets $A, B\in \mathcal{S}(\R^d)$, define the \textit{Hausdorff distance} by
\[
\d_H(A, B):=\inf\{\delta>0: \text{$A\subset B_\delta$ and $B\subset A_\delta$}\}.
\]
In other words, the Hausdorff distance between two sets indicates how much
both sets must be enlarged around their periphery in order to contain each
other. Informally, two sets are close in the Hausdorff distance if every point of either set is close to some point of the other set.

Define a self map $S$ on $(\mathcal{S}(\R^d),\d_H)$ by
\begin{equation}\label{zys}
S(E)=\bigcup_{i=1}^m S_i(E),\quad E\in \mathcal{S}(\R^d).
\end{equation}
For any set $E\subset \mathcal{S}(\R^d)$, it follows from Banach-fixed point theorem that $S^n(E)\to K$ in the Hausdorff distance as $n\to \infty,$ where $S^n$ is the $n$-th iterate of $S$ which is defined by \eqref{zys}. Furthermore, if $E\in \mathcal{S}(\R^d)$ satisfying
\[
\text{$S_i(E)\subset E$ for all $i=1,\ldots,m$,}
\]
then $S^{n+1}(E)\subset S^n(E)$ for all $n\ge 1.$ Hence, $\{S^n(E)\}$ forms a decreasing sequence of non-empty compact sets and the intersection is not empty. Write
 \[
 K'=\bigcap_{n=0}^\infty S^n(E).
 \]
We claim that $K'$ is nothing but the self-similar set $K$. In fact,
\[
S(K')=S\left(\bigcap_{n=0}^\infty S^n(E)\right)
     =\bigcap_{n=0}^\infty S(S^n(E))=\bigcap_{n=1}^\infty S^n(E)
     =K'.
\]
It follows from the uniqueness  that $K'=K.$  Thus,  the sets $\{S^n(E)\}$ provide increasingly good approximations to $K$, and the set $S^n(E)$ are sometimes called the \textit{$n$-generation} of $K$.

%

For an IFS $\{S_j\}_{j=1}^m$ of contractive similitudes, define
\[
\mathcal{C}_K=\bigcup_{i,j\in \{1,\ldots,m\}, i\not=j}(S_i(K)\cap S_j(K)), \quad \mathcal{C}=\Pi^{-1}(\mathcal{C}_K).
\]
Following \cite{Kig}, we say that $K$ is a \emph{post-critically finite (p.c.f. for short)} self-similar set if $\mathcal{P}=\cup_{n=1}^\infty \sigma^n(\mathcal{C})$ is a finite set, where $\sigma$ is the shift operator on $\Sigma_\infty$ such that $\sigma(i_1,i_2,\ldots,i_k,\ldots)=(i_2,\ldots,i_k,\ldots)$. Clearly, if a self-similar set satisfies the SSC, then $\mathcal{C}_K=\emptyset$  and hence it is a p.c.f. self-similar set. Any non-degenerate closed interval (i.e., not a singleton), the Sierpinski gasket and Vicsek set are p.c.f. self-similar sets. Moreover, the nested fractals defined by Lindstr{\o}m \cite{Lin} are p.c.f. self-similar sets.

We point out that the p.c.f. self-similar sets $K$ under consideration include many well-known fractals. Thus, the main results in this paper can be applied to these fractals. Here we give several examples.

\begin{Example}[Symmetric Cantor set]
 Let $S_1, S_2:\ \R\to \R$ be given by
\[
S_1(x)=\frac{1}{3}x~~~\text{and}~~~ S_2(x)=\frac{1}{3}x+\frac{2}{3}.
\]
The self-similar set $C$ of the IFS  $\{S_1, S_2\}$ satisfies $C=S_1(C)\cup S_2(C)$. The set $C$ is called the middle-third Cantor set and $\hdim (C)=\log 2/\log 3$. In fact, the middle-third Cantor set is a special symmetric Cantor set, which is the self-similar set defined as follows. Let $0<\lambda<\frac{1}{2}$ and define the IFS  $\{S_1, S_2\}$ by $S_1(x)=\lambda x, S_2(x)=\lambda x+(1-\lambda)$. The corresponding self-similar set is referred to as symmetric Cantor set, and its Hausdorff dimension is $\frac{\log 2}{-\log \lambda}.$
\end{Example}

\begin{figure}[!htb]
    \centering
\vspace{-0.8em}
    \includegraphics[width=10cm]{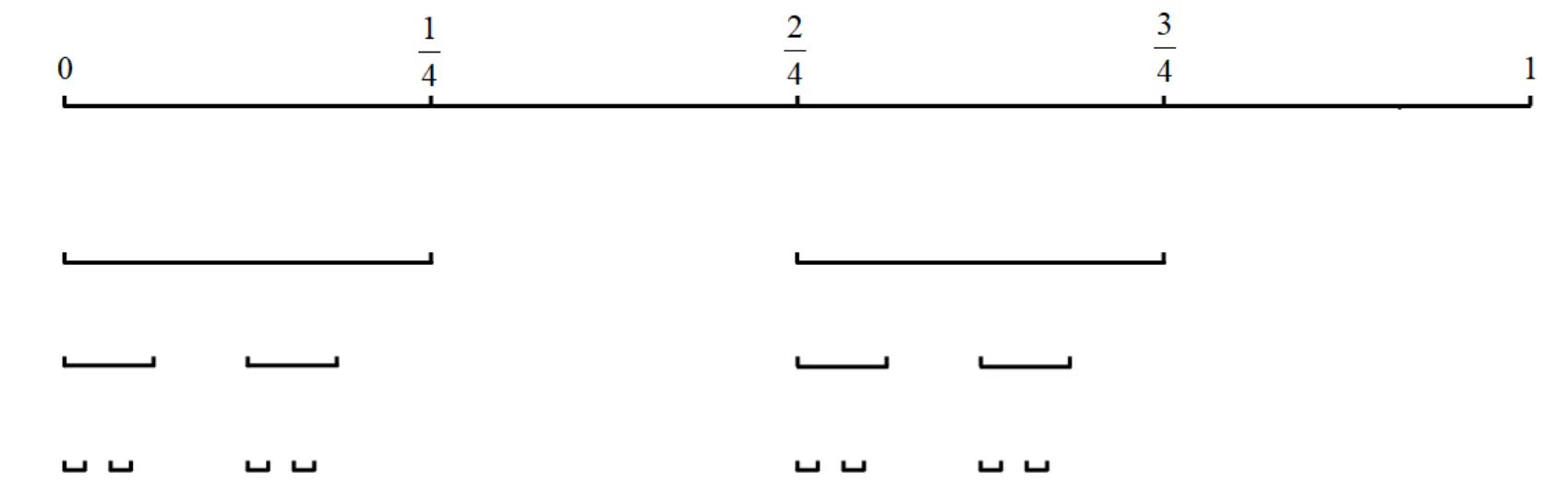}
\caption{Three generations of the middle-forth Cantor set.}
\end{figure}

\begin{Example}[Middle-forth Cantor set]
Let $S_1, S_2:\R\to \R$ be given by
\[
S_1(x)=\frac{1}{4}x~~~\text{and}~~~S_2(x)=\frac{1}{4}x+\frac{2}{4}.
\]
The corresponding self-similar set $K$ is called the middle-forth Cantor set and $s=\hdim (K)=1/2$.
\end{Example}

\begin{Example}[Generalized Sierpi\'{n}ski carpet]
Let $0<\lambda<\frac{1}{2}$ be a real number. Consider the IFS consisting of the contractions $\{S_i\}_{i=1}^4$  defined by
\[
S_1(x,y)=\lambda\begin{pmatrix}\!x\!\\[-1pt]\!y\!\end{pmatrix},~~ S_2(x,y)=\lambda\begin{pmatrix}\!x\!\\[-1pt]\!y\!\end{pmatrix}+\begin{pmatrix}\!1-\lambda\!\\[-1pt]\!0\!\end{pmatrix},
\]
and
\[
S_3(x,y)=\lambda\begin{pmatrix}\!x\!\\[-1pt]\!y\!\end{pmatrix}+\begin{pmatrix}\!0\!\\[-1pt]\!1-\lambda\!\end{pmatrix},~~
S_4(x,y)=\lambda\begin{pmatrix}\!x\!\\[-1pt]\!y\!\end{pmatrix}+\begin{pmatrix}\!1-\lambda\!\\[-1pt]\!1-\lambda\!\end{pmatrix}.
\]
The corresponding self-similar set $K$, called the generalized Sierpi\'{n}ski carpet, has Hausdorff dimension $s=\frac{2\log 2}{-\log \lambda}$. In fact, $K$ can be regarded as the product of two symmetric Cantor sets in $\R.$
\end{Example}

\begin{figure}[!htb]
    \centering
\makebox[\textwidth][c]{\includegraphics[width=10cm]{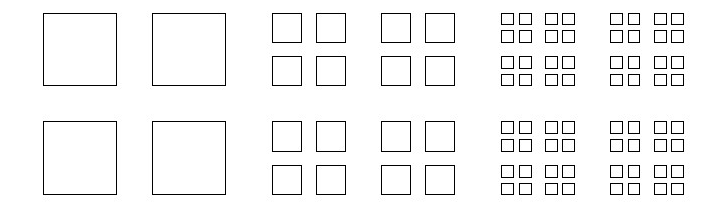}}
\caption{Three generations of the generalized Sierpi\'{n}ski carpet.}
\end{figure}
The above three examples satisfy the SSC, and the following three examples are p.c.f. self-similar sets which satisfy the \emph{open set condition} (OSC), i.e., there exists a non-empty bounded
open set $V$ such that $S_{i}(V)\cap S_{j}(V)=\emptyset$ for all $i\neq j$ and
$\bigcup_{i=1}^m S_{i}(V)\subset V.$
Clearly, the SSC implies the OSC.

\begin{Example}[Unit interval]\label{UI}
Let $I=[0,1]$. It is the self-similar set corresponding the IFS $\{S_i\}_{i=1}^2$ with $$S_1(x)=\frac{1}{2}x~~~\text{and}~~~S_2(x)=\frac{1}{2}(x+1).$$ Clearly, $\mathcal{C}_I=\{\frac{1}{2}\}$ and $\mathcal{C}=\Pi^{-1}(\mathcal{C}_I)=\{(1\dot{2}),(2\dot{1})\}$. So, $\mathcal{P}=\{(\dot{1}),(\dot{2})\}$. Here we write $\dot{s}=(s,s,\ldots)$. It is not difficult to check that a self-similar measure on $I$ is doubling if and only if it is the Lebesgue measure restricted on $I.$
\end{Example}

\begin{figure}[!htb]

\centering
\makebox[\textwidth][c]{\includegraphics[width=8cm]{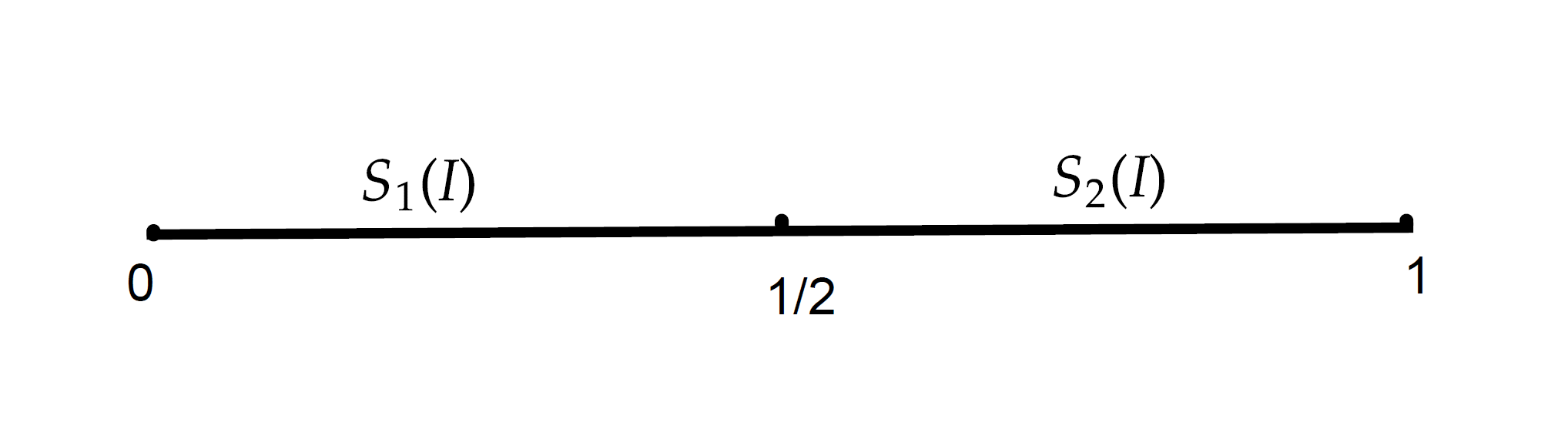}}
\caption{ The unit interval can be regarded as a self-similar set.}

\end{figure}

\begin{Example}[Sierpi\'{n}ski gasket]
The \textit{Sierpi\'{n}ski gasket} (also called Sierpi\'{n}ski triangle) is the self-similar set associated with the IFS consisting of the contractions $\{S_1, S_2,S_3\}$  defined by
\[
\text{$S_1(x,y)=\tfrac{1}{2}\begin{pmatrix}\!x\!\\[-1pt]\!y\!\end{pmatrix}, S_2(x,y)=\tfrac{1}{2}\begin{pmatrix}\!x\!\\[-1pt]\!y\!\end{pmatrix}+\begin{pmatrix}\!\tfrac{1}{2}\!\\[-1pt]\!0\!\end{pmatrix}$ and $S_3(x,y)=\tfrac{1}{2}\begin{pmatrix}\!x\!\\[-1pt]\!y\!\end{pmatrix}+\begin{pmatrix}\!\tfrac{1}{4}\!\\[-1pt]\!\tfrac{\sqrt{3}}{4}\!\end{pmatrix}$.}
\]
The Hausdorff dimension of the Sierpi\'{n}ski gasket is $s=\log 3/\log 2.$  A self-similar measure on the Sierpi\'{n}ski gasket is doubling if and only if it is the canonical one, i.e., the corresponding probability vector is $\textbf{p}=(1/3,1/3,1/3)$; see \cite[Proposition 1.3]{Yung}.
\end{Example}

\begin{figure}[!htb]

\centering
\makebox[\textwidth][c]{\includegraphics[width=10cm]{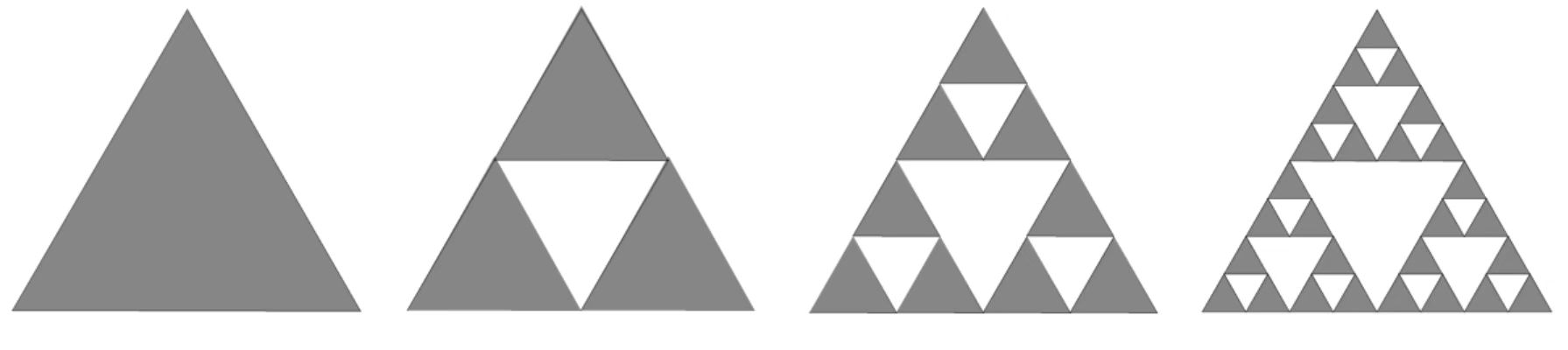}}
\caption{ The first three approximations of the Sierpi\'{n}ski gasket.}

\end{figure}

\begin{Example}[Vicsek set]
The \textit{Vicsek set} is the self-similar set associated with the IFS consisting of the contractions $\{S_i\}_{i=1}^5$  defined by
\[
S_i\begin{pmatrix}\!x\!\\[-1pt]\!y\!\end{pmatrix}=\frac{1}{3}\begin{pmatrix}\!x\!\\[-1pt]\!y\!\end{pmatrix}+\textbf{t}_i,
\]
where
\[
\textbf{t}_1=\begin{pmatrix}\!0\!\\[-1pt]\!0\!\end{pmatrix}, \textbf{t}_2=\begin{pmatrix}\!0\!\\[-1pt]\!\tfrac{2}{3}\!\end{pmatrix}, \textbf{t}_3=\begin{pmatrix}\!\tfrac{2}{3}\!\\[-1pt]\!0\!\end{pmatrix}, \textbf{t}_4=\begin{pmatrix}\!\tfrac{2}{3}\!\\[-1pt]\!\tfrac{2}{3}\!\end{pmatrix}, \textbf{t}_5=\begin{pmatrix}\!\tfrac{1}{3}\!\\[-1pt]\!\tfrac{1}{3}\!\end{pmatrix}.
\]
For the unit square $Q=[0,1]^2$ we have $S_i(Q)\subset Q$ for all $i=1,\ldots,5.$ The level-1, level-2, level-3 and level-4 cylinder squares of the Vicsek set are shown in Figure \ref{Vset}. The Hausdorff dimension of the Vicsek set is $s=\log 5/\log 3.$
\end{Example}

\begin{figure}[!htb]
\centering
\makebox[\textwidth][c]{\includegraphics[width=10cm]{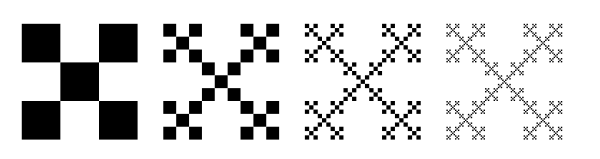}}
\caption{The first four approximations of the Vicsek set.}
\label{Vset}
\end{figure}

We collect some basic facts that we need in the sequel.
\begin{enumerate} [(1)]
\item  If $K$ is a self-similar set on $\R^d$ satisfying the SSC (or the OSC), then $\hdim (K)=s,$ where
\begin{equation}\label{ssd}
\sum_{i=1}^mr_i^s=1.
 \end{equation}
Here and in the sequel $\hdim (E)$ denotes the Hausdorff dimension of $E.$ We remark that a  p.c.f. self-similar set does not necessarily satisfy the OSC, see \cite{KT}.  Although generally there is no explicit formula for its Hausdorff dimension,  we always have $\hdim (K)\le s$ ($s$ is given by \eqref{ssd}) and many of them indeed satisfy the OSC (see \cite{DL}) and hence in this case their Hausdorff dimensions are exactly $s$.  For the definition and properties on Hausdorff dimension, we refer the reader to the famous book \cite{Fal}.
\item Given two finite sequences $\textbf{i}=(i_1, i_2,\ldots, i_n)$ and $\textbf{j}=(j_1,j_2,\ldots,j_m)$, we abbreviate $\textbf{i}\textbf{j}=(i_1, i_2,\ldots, i_n, j_1,j_2,\ldots,j_m)$. Two finite sequences $\textbf{i}$ and $\textbf{j}$ are \textit{incomparable} if each is not a prefix of the other, i.e., there is no $\textbf{k}\in \Sigma^*$ such that $\textbf{i}=\textbf{j}\textbf{k}$ or $\textbf{j}=\textbf{i}\textbf{k}$. For example, $\textbf{i}=(122)$ and $\textbf{j}=(11)$ are incomparable, but $\textbf{i}=(122)$ and $\textbf{j}=(12)$ are not incomparable. If $K$ is a p.c.f. self-similar set, then  $\mu(K_\textbf{i}\cap K_\textbf{j})=0$ if $\textbf{i}, \textbf{j}$ are incomparable. In particular, $\mu(S_i(F)\cap S_j(F))=0$ if $1\le i\not=j\le m.$
\item  If $K$ is a p.c.f. self-similar set, then $\mu(K_\textbf{i})=p_{\textbf{i}}:=p_{i_1}p_{i_2}\ldots p_{i_n}$ for any $\textbf{i}=i_1i_2\ldots i_n\in \Sigma^*.$
\end{enumerate}
For the proofs of $(2)$ and $(3)$, we refer the reader to \cite{Kig}.

The following lemma shows that the ``overlap part" of the p.c.f. self-similar set has zero Huasdorff content.
\begin{lemma}\label{pcfb}
Assume that $K$ is a p.c.f. self-similar set with Hausdorff dimension $s$ and  $\mu$ is a self-similar measure on it. For $0<\alpha\le s,$  we have $H^\alpha_\mu(S_\textbf{i}(K)\cap S_\textbf{j}(K))=0$ for any two distinct $\textbf{i}, \textbf{j}\in \Sigma_n.$
\end{lemma}
\begin{proof}
For two distinct $\textbf{i}, \textbf{j}\in \Sigma_n,$ it follows from \cite[Proposition 5.21]{Bar} that there exists a constant $C>0$ such that $\#(S_\textbf{i}(K)\cap S_\textbf{j}(K))\le C$. Here $\#(A)$ denotes the cardinality of the finite set $A$. Moreover, for any $x\in S_\textbf{i}(K)\cap S_\textbf{j}(K)$ there exists a decreasing basic cubes $\{K_{\textbf{i}_n}(x)\}_{n=1}^\infty\subset \mathcal{D}$ such that $\{x\}=\cap_{n=1}^\infty K_{\textbf{i}_n}(x).$ Hence,
\[
\text{$H^\alpha_\mu(S_\textbf{i}(K)\cap S_\textbf{j}(K))\le C \mu(K_{\textbf{i}_n}(x))^{\alpha/s}$ for any $n\ge 1$.}
\]
Since $\mu(K_{\textbf{i}_n}(x))=p_{\textbf{i}_n}\to 0$ as $n\to \infty$, the proof of the lemma is complete.
\end{proof}
The following example shows that Lemma \ref{pcfb} is not true for self-similar set with the OSC. Hence, we don't know whether Theorems \ref{MR1}, \ref{MR2} and \ref{lpc} are true for self-similar sets with the OSC.

\begin{Example}
Consider
\[
S_i\begin{pmatrix}\!x\!\\[-1pt]\!y\!\end{pmatrix}=\frac{1}{2}\begin{pmatrix}\!x\!\\[-1pt]\!y\!\end{pmatrix}+\textbf{t}_i,
\]
where
\[
\textbf{t}_1=\begin{pmatrix}\!0\!\\[-1pt]\!0\!\end{pmatrix}, \textbf{t}_2=\begin{pmatrix}\!0\!\\[-1pt]\!\tfrac{1}{2}\!\end{pmatrix}, \textbf{t}_3=\begin{pmatrix}\!\tfrac{1}{2}\!\\[-1pt]\!0\!\end{pmatrix}, \textbf{t}_4=\begin{pmatrix}\!\tfrac{1}{2}\!\\[-1pt]\!\tfrac{1}{2}\!\end{pmatrix}.
\]
It is easy to check that the IFS $\{S_i\}_{i=1}^4$ satisfies the OSC and the corresponding self-similar set is $K=[0,1]\times [0,1].$  Let $\mu$ be a self-similar measure on $K$ associated with the probability vector $(1/2,1/6,1/6,1/6)$. Clearly, $s=\hdim K=2$ and
\[
S_1(K)\cap S_2(K)=\left\{\frac{1}{2}\right\}\times \left[0,\frac{1}{2}\right].
\]
To cover $S_1(K)\cap S_2(K)$ by basic cubes in level $n,$ we need at least $2^{n-1}$ such basic cubes. Therefore,
\[
H^\alpha_\mu(S_1(K)\cap S_2(K))\ge 2^{n-1} \left(\frac{1}{6^n}\right)^{\alpha/s}=2^{n-1}\left(\frac{1}{6^n}\right)^{\alpha/2},
\]
which does not tend to zero if we choose proper $0<\alpha<2$ (e.g., $\alpha<2\log 2/\log 6$). This implies that $H^\alpha_\mu(S_1(K)\cap S_2(K))\not=0$ for some $\alpha.$
\end{Example}

\subsection{Hausdorff contents and Choquet integrals}
In this subsection, we show some lemmas on Hausdorff contents and Choquet integrals, which are required to show main theorems.
We call two basic cubes $Q,R\in \mathcal{D}$ \emph{essential disjoint} if $H_\mu^\alpha(Q\cap R)=0$. 

In the sequel, we will frequently use the following lemma.
\begin{lemma}\label{bjzl}
Assume that $0<\alpha\le s$ and $E\subset K$ with $H_\mu^\alpha(E)<\infty$. For any $\varepsilon>0,$ there exist a sequence of essential disjoint basic cubes $\{K_{\textbf{j}}\}_{\textbf{j}\in \Omega}$ with $\Omega\subset \Sigma^*$ such that $E\subset \bigcup_{\textbf{j}\in \Omega}K_{\textbf{j}}$ and
\[
\sum_{\textbf{j}\in \Omega}\mu(K_{\textbf{j}})^{\alpha/s}< H_\mu^\alpha(E)+\varepsilon.
\]
\end{lemma}
\begin{proof}
For any $\varepsilon>0,$ by the definition of the Hausdorff content there exists a sequence of basic cubes, denoted by $\{K_{\textbf{i}}\}_{\textbf{i}\in \Lambda}$, such that $E\subset \bigcup_{\textbf{i}\in \Lambda}K_{\textbf{i}}$ and \[
\sum_{\textbf{i}\in \Lambda}\mu(K_{\textbf{i}})^{\alpha/s}< H_\mu^\alpha(E)+\varepsilon.
 \]
It follows from  the construction of the basic cubes that if $Q, R\in \{K_{\textbf{i}}\}_{\textbf{i}\in \Lambda}$ and $Q\cap R\not=\emptyset$, then $Q\cap R$ has only finitely many points, or $Q\subset R$ or $Q\supset R$. Now, we take a subsequence from $\{K_{\textbf{i}}\}_{\textbf{i}\in \Lambda}$ in the following way.  If the intersection of two elements in $\{K_{\textbf{i}}\}_{\textbf{i}\in \Lambda}$ is nonempty and has only finitely many points, then we keep them; if one is contained in another then we delete the smaller one from $\{K_{\textbf{i}}\}_{\textbf{i}\in \Lambda}$.  Denote by $\{K_{\textbf{j}}\}_{\textbf{j}\in \Omega}$ the subsequence after the above operation. Clearly, it is disjoint and
\[
\sum_{\textbf{j}\in \Omega}\mu(K_{\textbf{j}})^{\alpha/s}\le \sum_{\textbf{i}\in \Lambda}\mu(K_{\textbf{i}})^{\alpha/s}< H_\mu^\alpha(E)+\varepsilon.
\]
The proof of the lemma is complete.
\end{proof}

The following result shows that the Hausdorff content $H_\mu^\alpha$ on self-similar set $K$ is strongly subadditive.
\begin{lemma}\label{ssuba}
Let $0<\alpha\le s$. For any sets $E_1, E_2\subset K,$ we have
\[
H_\mu^\alpha(E_1\cup E_2)+H_\mu^\alpha(E_1\cap E_2)\le H_\mu^\alpha(E_1)+H_\mu^\alpha(E_2).
\]
\end{lemma}
\begin{proof}
We may assume  that $H_\mu^\alpha(E_1)+H_\mu^\alpha(E_2)<\infty$. Fix $\varepsilon>0.$ It follows from Lemma \ref{bjzl} that there exist two sequences of essential disjoint  basic cubes $\{K_{\textbf{j}}\}_{\textbf{j}\in \Omega_1}$ and $\{K_{\textbf{i}}\}_{\textbf{i}\in \Omega_2}$ with $\Omega_1, \Omega_2\subset \Sigma^*$ such that $E_1\subset \bigcup_{\textbf{j}\in \Omega_1}K_{\textbf{j}}, E_2\subset \bigcup_{\textbf{i}\in \Omega_2}K_{\textbf{i}}$, and
\begin{equation}\label{cg}
\sum_{\textbf{j}\in \Omega_1}\mu(K_{\textbf{j}})^{\alpha/s}< H_\mu^\alpha(E_1)+\frac{\varepsilon}{2}, \quad   \sum_{\textbf{i}\in \Omega_2}\mu(K_{\textbf{i}})^{\alpha/s}< H_\mu^\alpha(E_2)+\frac{\varepsilon}{2}.
\end{equation}

 Let $\mathcal{B}$ be a maximal subfamily of the family $\{K_{\textbf{j}}, K_{\textbf{i}}\}_{\textbf{j}\in \Omega_1, \textbf{i}\in \Omega_2}$, in the sense that $E_1\cup E_2\subset \cup_{J\in \mathcal{B}} J$ and $I\cap J$ has at most finitely many points for any distinct $I,J\in \mathcal{B}$.
 Define
\[
 \mathcal{C}:=\{K_{\textbf{j}}\cap K_{\textbf{i}}: \textbf{j}\in \Omega_1, \textbf{i}\in \Omega_2\}.
\]
Clearly, $E_1\cap E_2\subset \bigcup_{Q\in \mathcal{C}}Q$.
 Note that $\mathcal{C}\subset \{K_{\textbf{j}}, K_{\textbf{i}}\}_{\textbf{j}\in \Omega_1, \textbf{i}\in \Omega_2}\cup\{\emptyset\}$ since the intersection of any two basic cubes is either has finitely many points or a smaller basic cube. Hence, noting $\mu$ is continuous we have

%

\[
\sum_{R\in\mathcal{B}}\mu(R)^{\alpha/s}+\sum_{Q\in\mathcal{C}}\mu(Q)^{\alpha/s}\le \sum_{\textbf{j}\in \Omega_1}\mu(K_{\textbf{j}})^{\alpha/s}+\sum_{\textbf{i}\in \Omega_2}\mu(K_{\textbf{i}})^{\alpha/s}.
\]
This, combined with  \eqref{cg}, implies that
\[
\begin{split}
H_\mu^\alpha(E_1\cup E_2)+H_\mu^\alpha(E_1\cap E_2)
&\le \sum_{R\in\mathcal{B}}\mu(R)^{\alpha/s}
+\sum_{Q\in\mathcal{C}}\mu(Q)^{\alpha/s}\\
&\le H_\mu^\alpha(E_1)+\frac{\varepsilon}{2}+H_\mu^\alpha(E_2)+\frac{\varepsilon}{2}\\
&=H_\mu^\alpha(E_1)+H_\mu^\alpha(E_2)+\varepsilon.
\end{split}
\]
Letting $\varepsilon\to 0$ completes the proof.
\end{proof}

With the help of Lemma \ref{bjzl}, we can show that the Hausdorff content $H_\mu^\alpha$ on self-similar set $K$ has the  upper semicontinuity.
\begin{proposition}\label{uppersc}
Let $0<\alpha\le s$. For any increasing sequence of sets $\{E_j\}_{j=1}^\infty\subset K$ with $E=\bigcup_{j=1}^\infty E_j$, we have
\[
H_\mu^\alpha(E)=\lim\limits_{j\to \infty}H_\mu^\alpha(E_j).
\]
\end{proposition}
\begin{proof}
Due to the monotonicity of $H_\mu^\alpha$, we only need to show that
\begin{equation}\label{dzx}
H_\mu^\alpha(E)\le \lim\limits_{j\to \infty}H_\mu^\alpha(E_j).
\end{equation}
We may assume that $\lim\limits_{j\to \infty}H_\mu^\alpha(E_j)<\infty$. Given $\varepsilon>0$, for any $j\ge 1$, by Lemma \ref{bjzl} there exist essential disjoint sequences $\{Q_{j,k}\}_{k}\subset \mathcal{D}$ such that $E_j\subset \bigcup_{k}Q_{j,k}$ and
\begin{equation}\label{jbfg}
\sum_k\mu(Q_{j,k})^{\alpha/s}<H_\mu^\alpha(E_j)+\frac{\varepsilon}{2^j}.
\end{equation}
Let $\{R_i\}_i$ be the maximal basic cubes of $\{Q_{j,k}\}_{j,k}$. Clearly,
\[
E=\bigcup_{j=1}^\infty E_j\subset \bigcup_{j=1}^\infty \bigcup_kQ_{j,k}= \bigcup_i R_{i}.
\]
Subsequently, we classify $\{R_i\}_i$ into $\{R_i^{(j)}\}_i$, $j=1,2,\ldots$, according to the following criteria:
Let $\{R_i^{(1)}\}_i\subset\{R_i\}_i$ be taken from $\{Q_{1,k}\}_k$, and for $j\ge 2$, let
$
\{R_i^{(j)}\}_i\subset\{R_i\}_i$ be taken from $\{Q_{j,k}\}_k$, but not from $\{Q_{1,k}, \cdots,Q_{j-1,k}\}_k$. Then $\{R_i^{(j_1)}\}_i\cap \{R_i^{(j_2)}\}_i=\emptyset$ for any $j_1\neq j_2$, and some of the sets $\{R_i^{(j)}\}_i$ maybe be empty, but that is harmless for the following argument. Clearly,
\begin{equation}\label{fg}
E=\bigcup_{j=1}^\infty E_j\subset \bigcup_{j=1}^\infty \bigcup_kQ_{j,k}=\bigcup_{j=1}^\infty\bigcup_i R_i^{(j)}.
\end{equation}
For any $j\in \mathbb N$, define
\begin{equation}\label{nds}
\{Q_{j,k}^*\}_k:=\{Q_{j,k}\}_k\setminus \{R_i^{(j)}\}_i.
\end{equation}

Fix a sufficient large positive integer $m$. Let $S_1:=\bigcup_i R_i^{(1)}\cap E_1$. Since $S_1\subset E_1\subset E_m\subset \bigcup_k Q_{m,k}$,  there exists a subsequence of $\{Q_{m,k}\}_k$, denoted by $\{Q^{(1)}_{m,k}\}_k$, such that $S_1\subset\bigcup_k Q^{(1)}_{m,k}$. We may assume each $Q^{(1)}_{m,k}$ intersects $S_1$. By the definition of  $H_\mu^\alpha,$ we have  $H_\mu^\alpha(E_1)\le \sum_{k} \mu(Q^{(1)}_{m,k})^{\alpha/s}$. Note that
\[
E_1=S_1\cup (E_1\setminus S_1)\subset \left(\bigcup_k Q^{(1)}_{m,k}\right)\bigcup \left(\bigcup_k Q^{*}_{1,k}\right).
\]
By the definition of $H_\mu^\alpha,$ we have
\begin{equation}\label{ineq1x}
H_\mu^\alpha(E_1)\le \sum_{k} \mu(Q^{(1)}_{m,k})^{\alpha/s}+\sum_{k} \mu(Q^{*}_{1,k})^{\alpha/s}.
\end{equation}
Hence, it follows from \eqref{jbfg} and the definition of $\{R_i^{(1)}\}_i$ that
\[
\begin{aligned}
\sum_i \mu(R_i^{(1)})^{\alpha/s}&=\sum_k \mu(Q_{1,k})^{\alpha/s}-\sum_k \mu(Q_{1,k}^*)^{\alpha/s}&& \text{(by \eqref{nds})}\\
&\le H_\mu^\alpha(E_1)+\frac{\varepsilon}{2}-\sum_k \mu(Q_{1,k}^*)^{\alpha/s} && \text{(by \eqref{jbfg})}\\
&\le\sum_{k} \mu(Q^{(1)}_{m,k})^{\alpha/s}+\frac{\varepsilon}{2}&& \text{(by \eqref{ineq1x})}.
\end{aligned}
 \]

For $j=2$, let $S_2:=\left(\bigcup_i R_i^{(2)}\cap E_2\right)\setminus S_1.$ It follows from $S_2\subset E_2\subset E_m\subset \bigcup_k Q_{m,k}$ that there exists a subsequence of $\{Q_{m,k}\}_k\setminus \{Q^{(1)}_{m,k}\}_k$, denoted by $\{Q^{(2)}_{m,k}\}_k$, such that $S_2\subset\bigcup_k Q^{(2)}_{m,k}$. Again,  we may assume each $Q^{(2)}_{m,k}$ intersects $S_2$. Observe that
\[
E_2=S_2\cup (E_2\setminus S_2)\subset \left(\bigcup_k Q^{(2)}_{m,k}\right)\bigcup \left(\bigcup_k Q^{*}_{2,k}\right)
\]
By the definition of $H_\mu^\alpha,$ we have
\begin{equation}\label{ineq1}
H_\mu^\alpha(E_2)\le \sum_{k} \mu(Q^{(2)}_{m,k})^{\alpha/s}+\sum_{k} \mu(Q^{*}_{2,k})^{\alpha/s}.
\end{equation}
This, together with \eqref{jbfg} and \eqref{nds}, further implies that
\[
\begin{aligned}
\sum_i \mu(R_i^{(2)})^{\alpha/s}&=\sum_k\mu(Q_{2,k})^{\alpha/s}-\sum_k\mu(Q^*_{2,k})^{\alpha/s} &&\text{(by \eqref{nds})}\\
&\le H_\mu^\alpha(E_2)-\sum_k\mu(Q^*_{2,k})^{\alpha/s} +\frac{\varepsilon}{2^2}&& \text{(by \eqref{jbfg})}\\
&\le \sum_{k} \mu(Q^{(2)}_{m,k})^{\alpha/s}+\frac{\varepsilon}{2^2}&&\text{(by \eqref{ineq1})}.
\end{aligned}
\]

In general, given any $3\le j\le m$ and $$S_j:=\left(\bigcup_i R_i^{(j)}\cap E_j\right)\setminus \left(\bigcup_{i=1}^{j-1}S_i\right),$$ there exists a subsequence of $$\{Q_{m,k}\}_k\setminus \left(\bigcup_{i=1}^{j-1}\{Q^{(i)}_{m,k}\}_k\right),$$ denoted by $\{Q^{(j)}_{m,k}\}_k$, such that $S_j\subset \bigcup_k Q_{m,k}^{(j)}$ and each $Q^{(j)}_{m,k}$ intersects $S_j$ as well as
\[
\sum_i \mu(R_i^{(j)})^{\alpha/s}\le \sum_{k} \mu(Q^{(j)}_{m,k})^{\alpha/s}+\frac{\varepsilon}{2^j}.
\]
This inequality, combined with \eqref{fg}, gives
\[
\begin{split}
\sum_{j=1}^m \sum_i \mu(R_i^{(j)})^{\alpha/s}&\le \sum_{j=1}^m\left(\sum_k \mu(Q^{(j)}_{m,k})^{\alpha/s}+\frac{\varepsilon}{2^j}\right)
\le \sum_k \mu(Q_{m,k})^{\alpha/s}+\varepsilon \\
&\le H_\mu^\alpha(E_m)+2\varepsilon\quad \text{(by $\eqref{jbfg}$)}.
\end{split}
\]
Letting $m\to \infty$ and combining with \eqref{fg}, we have
\[
H_\mu^\alpha(E)\le \sum_{j=1}^\infty \sum_i \mu(R_i^{(j)})^{\alpha/s}\le \lim\limits_{m\to \infty}H_\mu^\alpha(E_m)+2\varepsilon.
\]
Finally, we obtain \eqref{dzx} by letting $\varepsilon\to 0$ and hence complete the proof.
\end{proof}



Using Lemma \ref{ssuba} and Proposition \ref{uppersc}, we build the following sublinearity of Choquet integral on self-similar set $K$
following the method in \cite{STW}. We present the proof here for the completeness.

\begin{proposition}\label{subL}
Let $0<\alpha\le s$. For a sequence of nonnegative functions $\{f_j\}_{j=1}^\infty$ defined on $K$, we have
\begin{equation}\label{sublinearity}
\int_K \sum_{j=1}^\infty f_j(x) dH_\mu^\alpha\le \sum_{j=1}^\infty\int_K f_j(x)dH_\mu^\alpha.
\end{equation}
\end{proposition}
\begin{proof}
It follows immediately from the definition of Choquet integral that
\begin{equation}\label{xsx}
\text{$\int_K cf(x)dH_\mu^\alpha=c\int_K f(x)dH_\mu^\alpha\ $ for any $c\ge 0$ and  $f\ge 0$. }
\end{equation}
Moreover, if $0\le f_j \nearrow f$, then for any $t>0,$
\[
\{x\in K: f_j(x)>t\}\nearrow\{x\in K: f(x)>t\}.
\]
By Proposition \ref{uppersc} and the monotone convergence theorem for Lebesgue integral, we have
\begin{equation}\label{MCT}
\lim\limits_{j\to \infty}\int_Kf_j(x)dH_\mu^\alpha=\int_Kf(x)dH_\mu^\alpha.
\end{equation}
To prove \eqref{sublinearity}, it suffices to show that for any nonnegative functions $f$ and $g$ we have
\begin{equation}\label{ckj}
\int_K [f(x)+g(x)]dH_\mu^\alpha\le \int_K f(x)dH_\mu^\alpha+\int_K g(x)dH_\mu^\alpha
\end{equation}
due to \eqref{MCT}.
By \eqref{MCT} again and the simple function approximation theorem, we only need to verify \eqref{ckj} for simple functions $f$ and $g$ whose values are the form $\{\varepsilon N\}$, $N\in\N$, for some $\varepsilon>0$. By this and \eqref{xsx}, we only need to prove \eqref{ckj} for the functions $f$ and $g$ of the form (allowing multiplicity)
\[
\text{$f=\sum_{j=1}^m \chi_{F_j}$ and $g=\sum_{j=1}^m\chi_{G_j}$}
\]
with $m\in \N$. Here and in the sequel,  we let $\chi_A$ denote the indicator function of $A\subset K$.

When $f$ has this form, we write
\[
E_k=\left\{x\in K:\ \sum_{j=1}^m \chi_{F_j}(x)\ge k\right\},~~~k= 1,2,\ldots, m.
\]
Then we can check that $f=\sum_{k=1}^m \chi_{E_k}$; and by the definition of Choquet integral we have
\[
\int_K f(x)dH_\mu^\alpha=\sum_{j=1}^m\int_K \chi_{E_k}(x)dH_\mu^\alpha.
\]

Hence, to show \eqref{ckj}, we only need to prove
\begin{equation}\label{diq}
\sum_{j=1}^m\int_K  \chi_{B_j}(x) dH_\mu^\alpha\le \sum_{j=1}^m\int_K \chi_{A_j}(x)dH_\mu^\alpha
\end{equation}
for any finite sets $A_1,\ldots,A_m\subset K$ , where
\[
B_j=\left\{x\in K:\ \sum_{k=1}^m \chi_{A_k}(x)\ge j\right\},~~~j= 1,2,\ldots,m.
\]

We next show \eqref{diq} by induction on $m$. Indeed, for $m=2$, then $B_1=A_1\cup A_2$ and $B_2=A_1\cap A_2$. In this case, applying Lemma \ref{ssuba}, we obtain \eqref{diq}. We now assume that \eqref{diq} is true for $m=\ell$. Then, it follows that
\begin{align}\label{eqx1}
\sum_{j=1}^{\ell+1}\int_K \chi_{A_j}(x)dH_\mu^\alpha
\ge \sum_{j=1}^\ell\int_K  \chi_{B_j}(x) dH_\mu^\alpha
+\int_K \chi_{A_{\ell+1}}(x)dH_\mu^\alpha.
\end{align}
By Lemma \ref{ssuba}, we find that
\begin{align}\label{eqx2}
\int_K  \chi_{B_1}(x) dH_\mu^\alpha
+\int_K \chi_{A_{\ell+1}}(x)dH_\mu^\alpha
\ge
\int_K  \chi_{C_1}(x) dH_\mu^\alpha
+\int_K \chi_{D}(x)dH_\mu^\alpha
\end{align}
with
$C_1=B_1\cup A_{\ell+1}=\cup_{k=1}^{\ell+1}A_k$ and $D=B_1\cap A_{\ell+1}=\cup_{k=1}^{\ell}(A_k\cap A_{\ell+1})$.
Applying the assumption for $m=\ell$, we deduce that
\begin{align*}
\sum_{j=2}^\ell\int_K  \chi_{B_j}(x) dH_\mu^\alpha
+\int_K \chi_{D}(x)dH_\mu^\alpha
\ge \sum_{j=2}^{\ell+1}\int_K  \chi_{C_j}(x) dH_\mu^\alpha,
\end{align*}
where
\[
C_j=\left\{x\in K:\ \sum_{k=2}^m \chi_{B_k}(x)+\chi_{D}(x)\ge j-1\right\},~~~j= 2,\ldots,m+1.
\]
Combining this with \eqref{eqx1} and \eqref{eqx2}, we conclude that
$$\sum_{j=1}^{\ell+1}\int_K \chi_{A_j}(x)dH_\mu^\alpha\ge
\sum_{j=1}^{\ell+1}\int_K  \chi_{C_j}(x) dH_\mu^\alpha.$$
This finishes the proof of \eqref{diq} and hence of Proposition \ref{subL}.
\end{proof}

Now we establish a key estimation on the maximal operator for indicator function as follows.

\begin{lemma}\label{tzfgj}
Let $0<\alpha\leq s$ and $\frac{\alpha}{s}<p<\infty$. For any $K_{\textbf{i}}\in \mathcal{D}$, we have
\[
\int_K \left[M_{\mathcal{D}}^\mu(\chi_{K_{\textbf{i}}})(x)\right]^p dH_\mu^\alpha\le \frac{2p}{p-\frac{\alpha}{s}}\mu(K_{\textbf{i}})^{\frac{\alpha}{s}}.
\]
\end{lemma}
\begin{proof}
For $n\ge 1,$ define
\begin{equation}\label{dzxl}
\mathcal{D}_n=\bigcup_{k=0}^n\{K_{\textbf{i}}:\  \textbf{i}\in \Sigma_k\}.
\end{equation}
Clearly, $\mathcal{D}_n\nearrow \mathcal{D}$ as $n\to \infty.$ Let $M_{\mathcal{D}_n}^\mu$ be the Hardy--Littlewood maximal operator associated with $\mathcal{D}_n$, i.e.,
\[
M_{\mathcal{D}_n}^\mu f(x)=\sup_{1\le k\le n, \textbf{i}\in \Sigma_k}\frac{1}{\mu(K_{\textbf{i}}(x))}\int_{K_{\textbf{i}}(x)}|f(y)|d\mu(y).
\]
Then $M_{\mathcal{D}_n}^\mu\nearrow M_{\mathcal{D}}^\mu$ as $n\to \infty.$

For $K_{\textbf{i}}=K_{i_1i_2\ldots i_k}\in \mathcal{D}_n\subset \mathcal{D},$ it follows from the definition of the Hardy--Littlewood maximal operator that
\begin{equation}\label{tzbs}
M_{\mathcal{D}_n}^\mu (\chi_{K_{\textbf{i}}})(x)=\chi_{K_{\textbf{i}}}(x)+\sum_{j=1}^{k-1} a_j\chi_{K_{i_1\ldots i_{k-j}}\setminus K_{i_1\ldots i_{k-j+1}}}(x)+a_k\le 1,
\end{equation}
where
\[
a_j=\frac{\mu(K_\textbf{i})}{\mu(K_{i_1\ldots i_{k-j}})}\le 1, ~~~j=0,1,\ldots,k-1, ~~a_k=\mu(K_{\textbf{i}}).
\]
Clearly, $\{a_j\}$ is decreasing and we define  $a_{k+1}=0$ by convention.

By \eqref{tzbs} and the definition of the Choquet integral we have
\begin{equation}\label{cgj}
\begin{split}
\int_K \left
[M_{\mathcal{D}_n}^\mu (\chi_{K_{\textbf{i}}})(x)\right
]^pdH_{\mu}^\alpha&=p\int_0^1H_{\mu}^\alpha(\{M_{\mathcal{D}_n}^\mu (\chi_{K_{\textbf{i}}})(x)>t\})t^{p-1}dt\\
&\le p \int_0^1 \sum_{j=0}^k\mu(K_{i_1\ldots i_{k-j}})^{\alpha/s}\chi_{(a_{j+1},a_j]}(t)t^{p-1}dt.
\end{split}
\end{equation}
Now, let $\varepsilon=\frac{p-\alpha/s}{2}$. By the assumption we have $\varepsilon>0.$ It follows from \eqref{cgj} that
\begin{align*}
\int_K \left(M_{\mathcal{D}_n}^\mu (\chi_{K_{\textbf{i}}})\right)^pdH_{\mu}^\alpha&\le p\int_0^1\sum_{j=0}^k\mu(K_{i_1\ldots i_{k-j}})^{\alpha/s}a_j^{p-\varepsilon}\chi_{(a_{j+1},a_j]}(t)t^{\varepsilon-1}dt\\
&=p\int_0^1\sum_{j=0}^k\mu(K_{i_1\ldots i_{k-j}})^{-\varepsilon}\mu(K_{\textbf{i}})^{p-\varepsilon}\chi_{(a_{j+1},a_j]}(t)t^{\varepsilon-1}dt\\
&\le p\int_0^1\sum_{j=0}^k\mu(K_{\textbf{i}})^{-\varepsilon}\mu(K_{\textbf{i}})^{p-\varepsilon}\chi_{(a_{j+1},a_j]}(t)t^{\varepsilon-1}dt\\
&=p\mu(K_{\textbf{i}})^{\alpha/s}\int_0^1t^{\varepsilon-1}dt=\frac{2p}{p-\frac{\alpha}{s}}\mu(K_{\textbf{i}})^{\alpha/s}.
\end{align*}
Finally, combining the fact that $M_{\mathcal{D}_n}^\mu\nearrow M_{\mathcal{D}}^\mu$ as $n\to \infty$, \eqref{lp} and Proposition \ref{uppersc}, we complete the proof by letting $n\to \infty$.
\end{proof}

\section{Proofs of Theorems \ref{MR1} and \ref{MR2}}\label{s3}

In this section, we will prove Theorems \ref{MR1} and \ref{MR2}
and Corollaries \ref{weaki} and \ref{pp}.

\subsection{Proof of Theorem \ref{MR1}}
To prove Theorem \ref{MR1}, we can assume that $f\ge 0$ and $\int_K f(x)^pdH_\mu^\alpha<\infty.$ By the definition of Choquet integral, we have
\[
H_\mu^\alpha(\{x\in K:2^k<f(x)\le 2^{k+1}\})<\infty
\]
for any $k\in \Z$. By Lemma \ref{bjzl}, for any $k\in \Z$ and $\varepsilon>0$, there exists a sequence of essential disjoint basic cubes $\{Q_j^{(k)}\}_{j}\subset \mathcal{D}$  such that $\{x\in K: 2^k<f(x)\le 2^{k+1}\}\subset \bigcup_{j}Q_j^{(k)}$ and
\begin{equation}\label{sbds}
\sum_{j}\mu(Q_j^{(k)})^{\alpha/s}<
\begin{cases}
H_\mu^\alpha(\{x\in K: 2^k<f(x)\le 2^{k+1}\})+\frac{\varepsilon}{2^{k+(k+1)p}},\ &k\ge 0 \\
H_\mu^\alpha(\{x\in K: 2^k<f(x)\le 2^{k+1}\})+\frac{\varepsilon}{2^{-k+(k+1)p}},\ &k<0.
\end{cases}
\end{equation}
Define $g(x)=\sum_{k\in \Z} 2^{(k+1)p}\chi_{B_k}(x)$, where $B_k=\bigcup_{j}Q_j^{(k)}$ and $x\in K$. It is easy to check that $f^p\le g.$ We next consider two cases: $p\ge 1$ and $\alpha/s<p<1.$

\textsc{Case 1.} $p\ge 1.$  By H\"{o}lder's inequality and the fact that $f^p\le g$, we find that, for any $x\in K$,
\begin{align*}
[M_{\mathcal{D}}^\mu f(x)]^p&\le M_{\mathcal{D}}^\mu (f^p)(x)\qquad \text{(by H\"{o}lder's inequality)}\\
&\le M_{\mathcal{D}}^\mu (g)(x)\qquad\quad \text{(since $f^p\le g$)}\\
&\le \sum_{k\in \Z}2^{(k+1)p}\sum_jM_{\mathcal{D}}^\mu(\chi_{Q_j^{(k)}})(x)~~~ \text{(by sublinearity of $M_{\mathcal{D}}^\mu$)}.
\end{align*}
Hence
\begin{align*}
&\int_K [M_{\mathcal{D}}^\mu f(x)]^pdH_\mu^\alpha\\
&\quad\le \sum_{k\in \Z} 2^{(k+1)p}\sum_j\int_K M_{\mathcal{D}}^\mu(\chi_{Q_j^{(k)}})(x)dH_\mu^\alpha \quad \text{(by Proposition \ref{subL})}\\
&\quad\le \frac{2}{1-\frac{\alpha}{s}} \sum_{k\in \Z} 2^{(k+1)p}\sum_j \mu(Q_j^{(k)})^{\alpha/s}\quad \text{(by Lemma \ref{tzfgj})}\\
&\quad\le \frac{2}{1-\frac{\alpha}{s}}\left(\sum_{k\in \Z} 2^{(k+1)p} H_\mu^\alpha(\{x:2^k<f(x)\le 2^{k+1}\})+\sum_{k\ge 0}\frac{\varepsilon}{2^k}+\sum_{k< 0}2^k\varepsilon\right)\quad \text{(by \eqref{sbds})}\\
&\quad\le \frac{2}{1-\frac{\alpha}{s}}2^{p+1}\sum_{k\in \Z}\int_{2^{k-1}}^{2^k}(2^k)^{p-1}dt H_\mu^\alpha(\{x:f(x)>2^k\})+\frac{6\varepsilon}{1-\frac{\alpha}{s}}\\
&\quad\le \frac{2}{1-\frac{\alpha}{s}}2^{p+1}\int_0^\infty \sum_{k\in\Z} H_\mu^\alpha(\{x:f(x)>2^k\})\chi_{(2^{k-1},2^k]}(t)(2t)^{p-1}dt+\frac{6\varepsilon}{1-\frac{\alpha}{s}}\\
&\quad\le\frac{2^{2p+1}}{p(1-\frac{\alpha}{s})} \left(p\int_0^\infty t^{p-1}H_\mu^\alpha(\{x:f(x)>t\})\sum_{k\in\Z}\chi_{(2^{k-1},2^k]}(t)dt\right)+\frac{6\varepsilon}{1-\frac{\alpha}{s}}\\
&\quad=\frac{2^{2p+1}}{p(1-\frac{\alpha}{s})}\int_K |f(x)|^p dH_\mu^\alpha+\frac{6\varepsilon}{1-\frac{\alpha}{s}}.
\end{align*}
The last inequality holds since $\{x:f(x)>2^k\}\subset \{x:f(x)>t\}$ when $t\in (2^{k-1},2^k]$ and $\sum_{k\in \Z}\chi_{(2^{k-1},2^k]}(t)=1$ when $t\in (0,\infty)$. Taking $\varepsilon\to 0$, we have
\begin{align*}
\int_K [M_{\mathcal{D}}^\mu f(x)]^pdH_\mu^\alpha
\le\frac{2^{2p+1}}{p(1-\frac{\alpha}{s})}\int_K |f(x)|^p dH_\mu^\alpha.
\end{align*}

\textsc{Case 2.} $\alpha/s<p<1.$ Noting that $f\le \sum_{k\in \Z} 2^{k+1}\chi_{B_k}$, we have
\[
M_{\mathcal{D}}^\mu f\le \sum_{k\in \Z}2^{k+1}\sum_jM_{\mathcal{D}}^\mu(\chi_{Q_j^{(k)}}).
\]
Observe that $(\sum_i a_i)^p\le \sum a_i^p$ for any nonnegative sequence $\{a_j\}$. We have
\[
\left(M_{\mathcal{D}}^\mu f\right)^p\le \sum_{k\in \Z}2^{(k+1)p}\sum_j\left[M_{\mathcal{D}}^\mu(\chi_{Q_j^{(k)}})\right]^p.
\]
Therefore, it follows that
\begin{align*}
&\int_K \left[M_{\mathcal{D}}^\mu f\right(x)]^pdH_\mu^\alpha\\
&\quad\le \frac{2p}{p-\frac{\alpha}{s}}\sum_{k\in \Z}2^{(k+1)p}\sum_j \mu(Q_j^{(k)})^{\alpha/s}~~\text{(by Proposition \ref{subL} and Lemma \ref{tzfgj})}\\
&\quad\le\frac{2p}{p-\frac{\alpha}{s}}\sum_{k\in \Z} 2^{(k+1)p} H_\mu^\alpha(\{x:2^k<f(x)\le 2^{k+1}\})+\frac{6p\varepsilon}{p-\frac{\alpha}{s}}\quad \text{(by \eqref{sbds})}\\
&\quad\le \frac{2p}{p-\frac{\alpha}{s}}2^{p+1}\sum_{k\in \Z}\int_{2^{k-1}}^{2^k}(2^k)^{p-1}dt H_\mu^\alpha(\{x:f(x)>2^k\})+\frac{6p\varepsilon}{p-\frac{\alpha}{s}}\\
&\quad\le\frac{2p}{p-\frac{\alpha}{s}}2^{p+1}\int_0^\infty \sum_{k\in\Z} H_\mu^\alpha(\{x:f(x)>2^k\})\chi_{(2^{k-1},2^k]}(t)t^{p-1}dt+\frac{6p\varepsilon}{p-\frac{\alpha}{s}}\\
&\quad\le \frac{2}{p-\frac{\alpha}{s}}2^{p+1}\left(p\int_0^\infty t^{p-1}H_\mu^\alpha(\{x:f(x)>t\})\sum_{k\in\Z}\chi_{(2^{k-1},2^k]}(t)dt\right)+\frac{6p\varepsilon}{p-\frac{\alpha}{s}}\\
&\quad=\frac{2^{p+2}}{p-\frac{\alpha}{s}}\int_K |f(x)|^p dH_\mu^\alpha+\frac{6p\varepsilon}{p-\frac{\alpha}{s}}.
\end{align*}
Taking $\varepsilon\to 0$, we have
\begin{align*}
\int_K [M_{\mathcal{D}}^\mu f(x)]^pdH_\mu^\alpha
\le\frac{2^{p+2}}{p-\frac{\alpha}{s}}\int_K |f(x)|^p dH_\mu^\alpha.
\end{align*}
Combining the above two cases, we complete the proof of Theorem \ref{MR1}.

\subsection{Proof of Theorem \ref{MR2} }
To prove Theorem \ref{MR2}, we require the following
elegant lemma, which overcomes the difficulty of the lack of linearity in Choquet integral, that is, inequality \eqref{eq11-11-a} below is now accessible.
Its proof is essentially inspired by \cite[Lemma 2]{OV}. The dyadic decomposition technique in \cite{OV}, however, is no longer applicable on fractal sets, and we must use basic construction cubes in fractal geometry in place of dyadic cubes to complete the proof.
 We point out that,
in the general case, \eqref{eq11-11-a} is not true for non-overlapping dyadic cubes if $0<\alpha<s$.

\begin{lemma}\label{zfg}
Let $0<\alpha\le s$ and $\{I_j\}_{j}\subset \mathcal{D}$ be a family of essential disjoint basic cubes. Then there exists a subfamily $\{I_{j_k}\}_{k}$ that satisfies the following:
\begin{enumerate}
\item [\textup{(i)}]$\sum\limits_{I_{j_k}\subset I}\mu(I_{j_k})^{\alpha/s}\le 2\mu(I)^{\alpha/s}$ for each basic cube $I\in \mathcal{D};$
\item [\textup{(ii)}]$H_{\mu}^\alpha(\bigcup\limits_j I_j)\le 2\sum\limits_k \mu(I_{j_k})^{\alpha/s};$
\item [\textup{(iii)}] for any $f\in L^1(\cup_{k} I_{j_k},H^{\alpha}_\mu)$,
we have
\begin{equation}\label{eq11-11-a}
\sum_{k}\int_{I_{j_k}}|f(x)|\,dH_{\mu}^\alpha
\le 2\int_{\bigcup\limits_{k} I_{j_k}}|f(x)|\,dH_{\mu}^\alpha.
\end{equation}
\end{enumerate}
\end{lemma}

\begin{proof}
$\textup{(i)}$ We choose the desired subsequence $\{I_{j_k}\}_{k}$ by induction. First, let $j_1=1$. If $j_1,\ldots,j_k$ has been chosen so that $(i)$ holds, then let $j_{k+1}$ be the first index (if that exists) such that the family $\{I_{j_1}, \ldots,I_{j_{k+1}}\}$ satisfies $\textup{(i)}$. Therefore property (i) holds.

$\textup{(ii)}$ For $k\ge 1,$ consider index $j$ with $j_k<j<j_{k+1}$. It follows from the definition of $j_{k+1}$ that there exists a basic cube $I_j^*\supset I_j\bigcup (\cup_{\ell=1}^kI_{j_\ell})$ such that
\[
\sum_{I_{j_\ell}\subset I_j^*, \ell\le k}\mu(I_{j_\ell})^{\alpha/s}+\mu(I_j)^{\alpha/s}>2\mu(I^*_j)^{\alpha/s}.
\]
This, combined with the fact that $\mu(I_j^*)^{\alpha/s}\ge \mu(I_j)^{\alpha/s}$, gives
\begin{equation}\label{3e21}
\mu(I^*_j)^{\alpha/s}\le \sum_{I_{j_\ell}\subset I_j^*, \ell\le k}\mu(I_{j_\ell})^{\alpha/s}.
\end{equation}
Let $\{\widetilde{I_v}\}$ be the maximal basic cubes of $\{I_j^*\}_j$ such that $j_{k_v}<v<j_{k_v+1}$. Then
\[
\bigcup_{j}I_j\subset \left(\bigcup_{k}I_{j_k}\right)\cup \left(\bigcup_{v}\widetilde{I_v}\right).
\]
This, together with \eqref{3e21}, further implies that
\begin{align*}
H_\mu^\alpha\left(\bigcup_{j}I_j\right)\le \sum_{k}\mu(I_{j_k})^{\alpha/s}+\sum_{v}\mu(\widetilde{I_v})^{\alpha/s}\le \sum_{k}\mu(I_{j_k})^{\alpha/s}+\sum_v\sum_{I_{v_\ell}\subset \widetilde{I_v}, \ell\le k_v}\mu(I_{v_\ell})^{\alpha/s}\le 2\sum_{k}\mu(I_{j_k})^{\alpha/s}.
\end{align*}
The proof of the second statement is complete.

$\textup{(iii)}$ Finally, via the packing condition \textup{(i)}, we show
\eqref{eq11-11-a}. It suffices to prove, for any $t\in(0,\infty)$,
\begin{align*}
\sum_{k}H_\mu^\alpha(\{x\in I_{j_k}:\ |f(x)|>t\})
\le C H_\mu^\alpha(\{x\in \bigcup_{k}I_{j_k}:\ |f(x)|>t\}).
\end{align*}
To this end, without loss of generality, assume for any $t\in(0,\infty)$,
$$H_\mu^\alpha(\{x\in\bigcup_{k} I_{j_k}:\ |f(x)|>t\})>0.$$
Otherwise, the conclusion \eqref{eq11-11-a} follows directly.
Then applying the definition of Hausdorff contents, we know that there exists a family of essential disjoint basic cubes $\{P_i\}_{i}$ such that
\[
\bigcup_{k}\{x\in I_{j_k}:\
|f(x)|>t\}=\{x\in\bigcup_{k} I_{j_k}:\
|f(x)|>t\}\subset \bigcup_i P_i
\]
and, for arbitrarily small number $\varepsilon>0$,
\begin{equation}\label{eq11-11-x}
\sum_i \mu(P_i)^{\alpha/s}
\le(1+\varepsilon) H_\mu^\alpha(\{x\in\bigcup_{k} I_{j_k}:\
|f(x)|>t\}).
\end{equation}
Moreover, we may assume that, for any $i$, $\{x\in\bigcup_{k} I_{j_k}:\
|f(x)|>t\}\cap P_i\neq \emptyset$.

Let $A_1$ be the set of all $i$ such that there exists only one index $j_i^*$
 satisfying
 \[\{x\in I_{j_k^*}:\ |f(x)|>t\}\subset P_i,\]
and $A_2$ the set of all $i$ such that there exist at least two indices $\{i_v\}_v$ satisfying
\[\bigcup_v\{x\in I_{j_{i_v}}:\ |f(x)|>t\}\subset P_i.\]
Moreover, let $A$ be the set of all $j_k$ such that
the subset $\{x\in I_{j_k}:\ |f(x)|>t\}$ is covered
by at least two cubes from $\{P_i\}_i$.
Notice that
\begin{align}\label{eq12-5-x}
\begin{split}
&\sum_{k}H_\mu^\alpha(\{x\in I_{j_k}:\ |f(x)|>t\})\\
&\quad=\sum_{i\in A_1}H_\mu^\alpha(\{x\in I_{j_i^*}:\ |f(x)|>t\})+\sum_{i\in A_2}\sum_{v}H_\mu^\alpha(\{x\in I_{j_{i_v}}:\ |f(x)|>t\})\\
&\quad\quad+\sum_{j_k\in A}H_\mu^\alpha(\{x\in I_{j_k}:\ |f(x)|>t\}).
\end{split}
\end{align}
Obviously,
\begin{equation}\label{eq11-11-c}
\sum_{i\in A_1}H_\mu^\alpha(\{x\in I_{j_i^*}:\ |f(x)|>t\})
\leq \sum_{i\in A_1}\mu(P_i)^{\alpha/s}.
\end{equation}
In addition, note that $\bigcup_{v}I_{j_{i_v}}\subset P_i$. By this and the packing condition \textup{(i)},
we have, for any $i\in A_2$,
\[\sum_{v}\mu(I_{j_{i_v}})^{\alpha/s}\le 2\mu(P_i)^{\alpha/s},\]
which further implies that
\begin{equation}\label{eq11-11-d}
\sum_{i\in A_2}\sum_{v}H_\mu^\alpha(\{x\in I_{j_{i_v}}:\ |f(x)|>t\})
\le 2\sum_{i\in A_2} \mu(P_i)^{\alpha/s}.
\end{equation}
For any $j_k\in A$, we let $\{P_{j_k}^v\}_v$ be a sub-family of $\{P_i\}_i$ such that
$$\{x\in I_{j_k}:\ |f(x)|>t\}\subset \bigcup_ v P_{j_k}^v.$$
Then
\begin{equation}\label{eq11-11-e}
\sum_{j_k\in A}H_\mu^\alpha(\{x\in  I_{j_k}:\ |f(x)|>t\})
\leq \sum_{j_k\in A}\sum_v\mu(P_{j_k}^v)^{\alpha/s}.
\end{equation}
Therefore, by \eqref{eq12-5-x}, \eqref{eq11-11-c}, \eqref{eq11-11-d} and \eqref{eq11-11-e}, we infer that
\begin{align*}
&\sum_{k}H_\mu^\alpha(\{x\in I_{j_k}:\ |f(x)|>t\})\le \sum_{i\in A_1}\mu(P_i)^{\alpha/s}+2\sum_{i\in A_2}\mu(P_i)^{\alpha/s}
+\sum_{j_k\in A}\sum_v \mu(P_{j_k}^v)^{\alpha/s}\le 2\sum_i\mu(P_i)^{\alpha/s}.
\end{align*}
From this and \eqref{eq11-11-x}, we deduce that
\begin{equation*}
\sum_{k}H_\mu^\alpha(\{x\in I_{j_k}:\ |f(x)|>t\})
\le 2(1+\varepsilon)H_\mu^\alpha
(\{x\in\bigcup_{k} I_{j_k}:\ |f(x)|>t\}).
\end{equation*}
Combining this and the definition of Choquet
integrals, we further conclude that
\[\begin{aligned}
\sum_{k} \int_{I_{j_k}}|f(x)|\,dH_\mu^\alpha
&=\sum_{k}\int^{\infty}_{0}H_\mu^\alpha(\{x\in I_{j_k}:\ |f(x)|>t\})\,dt\\
&=\int^{\infty}_{0}\sum_{k}H_\mu^\alpha(\{x\in I_{j_k}:\ |f(x)|>t\})\,dt\\
&\le 2(1+\varepsilon)\int^{\infty}_{0}H_\mu^\alpha(\{x\in \bigcup_{k}I_{j_k}:\ |f(x)|>t\})\,dt\\
&\le2(1+\varepsilon)\int_{\bigcup_{k} I_{j_k}}|f(x)|\,dH_\mu^\alpha,
\end{aligned}\]
which is the desired inequality \eqref{eq11-11-a} by taking $\varepsilon\to 0$. This completes the proof of Lemma \ref{zfg}.
\end{proof}

\begin{remark}
Lemma \ref{zfg}\textup{(i)} is always referred to as the packing condition, which guarantees the validity of the key inequality \eqref{eq11-11-a}.
\end{remark}

The next lemma shows that the $s$-dimensional Hausdorff content coincides with the self-similar measure $\mu.$

\begin{lemma}\label{eq}
For any Borel set $E\subset K,$ we have $H_\mu^s(E)=\mu(E)$.
\end{lemma}
\begin{proof}
For any Borel set $E\subset K,$ recall that when $\alpha=s$ we have
\[
H_\mu^s(E)=\inf\left\{\sum_{j}\mu(I_j):\ E\subset \bigcup_{j}I_j,\ I_j\in \mathcal{D}\right\}.
\]
It is easy to check that $H_\mu^s$ is an outer measure on $K$ and $H_\mu^s(I_j)=\mu(I_j)$ for any $I_j\in \mathcal{D}$. Since $\mathcal{D}$ is an semi-algebra on $K$, by Carath\'{e}odory's extension theorem (see, for example, Theorem 1.14 in \cite{Foll}) we have $H_\mu^s(E)=\mu(E)$ for any Borel subset $E\subset K.$
%
\end{proof}

\begin{lemma}\label{jfgj}
Assume that $0<\alpha\le s$ and $f\ge0$ is a $\mu$-measurable function. Then, for any basic cube $I\in \mathcal{D}$ we have
\[
\int_I f(x)d\mu(x)\le \frac{s}{\alpha}\left(\int_I f(x)^{\alpha/s}d H_{\mu}^\alpha\right)^{s/\alpha}.
\]
\end{lemma}
\begin{proof}
Using the elementary inequality $\sum_{j=1}^\infty a_j\le \left(\sum_{j=1}^\infty a_j^q\right)^{1/q}$, where $0<q\le 1$ and $a_j>0,$ we have
\[
\left(\sum_{j=1}^\infty\mu(I_j)\right)^{1/s}\le \left(\sum_{j=1}^\infty\mu(I_j)^{\alpha/s}\right)^{1/\alpha}
\]
for any $I_j\in \mathcal{D}.$ Hence, by the definition of content we have $(H^s_\mu(E))^{1/s}\le (H^\alpha_\mu(E))^{1/\alpha}$ for any subset $E\subset K.$ If we write $\lambda_\beta(t)=H^\beta_\mu\{x\in I:f(x)>t\}$. Then
\[
(\lambda_s(t))^{1/s}\le (\lambda_\alpha(t))^{1/\alpha}, ~~t>0.
\]
From this and Lemma \ref{eq}, we infer that
\begin{align*}
\int_I f(x)d\mu(x)&=\int_I f(x)d H_{\mu}^\alpha=\int_0^\infty \lambda_s(t)dt\le \frac{s}{\alpha}\int_0^\infty \lambda_\alpha(r^{s/\alpha})^{s/\alpha}r^{(s/\alpha-1)}dr\\
&\le \frac{s}{\alpha}\left(\int_I f^{\alpha/s}d H^\alpha_\mu\right)^{(s/\alpha-1)}\int_0^\infty \lambda_\alpha(r^{s/\alpha})dr
\le \frac{s}{\alpha}\left(\int_I f^{\alpha/s}d H^\alpha_\mu\right)^{s/\alpha},
\end{align*}
where in the second inequality we used the fact that
\[
\lambda_\alpha(r^{s/\alpha})r\le \int_{\{x\in K: f^{\alpha/s}(x)>r\}}f^{\alpha/s}dH^\alpha, ~~r>0.
\]
The proof of the lemma is complete.
\end{proof}

Now, we are in a position to prove Theorem \ref{MR2}.
\begin{proof}[Proof of Theorem \ref{MR2}]
We can assume that $f\ge 0.$ For $n\ge 1,$ let $\mathcal{D}_n$ be defined by \eqref{dzxl}. Again, $M_{\mathcal{D}_n}^\mu\nearrow M_{\mathcal{D}}^\mu$ as $n\to \infty,$ where $M_{\mathcal{D}_n}^\mu$ be the Hardy--Littlewood maximal operator associated with $\mathcal{D}_n$. Write $\mathcal{B}_{n,t}=\{I\in \mathcal{D}_n: \frac{1}{\mu(I)}\int_I fd\mu>t\}$. Clearly,
\[
\{x\in K: M_{\mathcal{D}_n}^\mu f(x)>t\}=\bigcup\{I: I\in \mathcal{B}_{n,t}\}.
\]
Let $\mathcal{B}_{n,t}^*$ be the maximal cubes in $\mathcal{B}_{n,t}$. Then they are essential disjoint and
\[
\bigcup\{I:I\in \mathcal{B}_{n,t}^*\}=\bigcup\{I: I\in \mathcal{B}_{n,t}\}.
\]
For any  $I\in \mathcal{B}_{n,t}^*,$ it follows from Lemma \ref{jfgj} that
\[
\mu(I)^{\alpha/s}\le \left(\frac{1}{t}\int_I f(x)d\mu\right)^{\alpha/s}\le \left(\frac{s}{\alpha}\right)^{\alpha/s}t^{-\alpha/s}\int_I[f(x)]^{\alpha/s}dH_\mu^\alpha.
\]
Applying Lemma \ref{zfg} to the family $\mathcal{B}_{n,t}^*,$ we obtain a subfamily $\{I_{j_k}\}_k$ such that
\begin{align*}
H_\mu^\alpha(\{x\in K: M_{\mathcal{D}_n}^\mu f(x)>t\})
&=H_\mu^\alpha\left(\bigcup\{I:I\in \mathcal{B}_{n,t}^*\}\right)\\
&\le 2\sum_k\mu(I_{j_k})^{\alpha/s}\quad \text{(by Lemma \ref{zfg}(ii))}\\
&\le 2\left(\frac{s}{\alpha}\right)^{\alpha/s}t^{-\alpha/s}\sum_k\int_{I_{j_k}}[f(x)]^{\alpha/s}dH_\mu^\alpha\\
&\le  4\left(\frac{s}{\alpha}\right)^{\alpha/s}t^{-\alpha/s}\int_K [f(x)]^{\alpha/s}dH_\mu^\alpha\quad \text{(by Lemma \ref{zfg}(iii))}.
\end{align*}
Finally, the desired result follows from the fact that the left side converges to $H_\mu^\alpha(\{x\in K: M_\mathcal{D}^\mu f(x)>t\})$ as $n\to \infty$ and Proposition \ref{uppersc}.
\end{proof}

We should point out that in the proofs of Theorems \ref{MR1} and \ref{MR2} we follow the strategy of
 Orobitg and Verdera \cite{OV}. In addition, recall that Liu \cite{Liu} gave some similar estimates for Hardy--Littlewood maximal operators on doubling metric measure spaces without quantitative constants, and Saito et al. \cite{STW} considered these estimates on ${\mathbb R}^d$.

At the end of this subsection, we give proofs of Corollaries \ref{weaki} and \ref{pp}.

\begin{proof}[Proof of Corollary \ref{weaki}]
It follows immediately from Theorem \ref{MR2} with $\alpha=s$ and  Lemma \ref{eq}.
\end{proof}

\begin{proof}[Proof of Corollary \ref{pp}]
Let $L^{\infty}(K,\mu)$ denote the space of all functions $g$ on $K$ such that the quasi-norm
$$\|g\|_{L^\infty(K,\mu)}:=\inf_{E_0\subset K,\,\mu(E_0)=0}\sup_{x\in K\backslash E_0}|g(x)|<\infty.$$
Then we have $\|M_{\mathcal{D}}^\mu g\|_{L^{\infty}(K,\mu)}
\le \|g\|_{L^{\infty}(K,\mu)}$ for any $g\in L^{\infty}(K,\mu)$.
For $f\in L^p(K,\mu)$, $1<p<\infty$, and any given $0<s<\infty$, let $$f_1:=f\chi_{\{x\in K:\,|f(x)|>s\}}~~~\text{and}~~~f_2:=f-f_1=f\chi_{\{x\in K:\,|f(x)|\le s\}}.$$
It is easy to see that $f_1\in L^{1}(K,\mu)$ and $f_2\in L^{\infty}(K,\mu)$. Then by Corollary \ref{weaki} and a standard argument, we obtain
Corollary \ref{pp}.
\end{proof}

\section{Proofs of Theorem \ref{lpc} and Propositions \ref{MR3} and \ref{t16}}\label{s4}

\subsection{Proof of Theorem \ref{lpc}}
The importance of the Hardy--Littlewood maxinal operator lies in the fact that it majorizes many important operators in analysis.
Here, applying Corollary \ref{weaki}, we give the Lebesgue differentiation theorem on the self-similar set $K$.

\begin{theorem}\label{A2}
Let $f\in L^1_{\rm loc}(K,\mu)$. Then

\begin{enumerate}
 \item [\textup{(i)}]for $\mu$-almost every $x\in K$,
\[\lim_{n\to \infty}\frac1{\mu(K_{\mathbf{i}_n}(x))}\int_{K_{\mathbf{i}_n}(x)}f(y)d\mu=f(x);
\]
 \item [\textup{(ii)}]for $\mu$-almost every $x\in K$,
\[\lim_{n\to \infty}\frac1{\mu(K_{\mathbf{i}_n}(x))}\int_{K_{\mathbf{i}_n}(x)}|f(y)-f(x)|d\mu=0.
\]
\end{enumerate}
\end{theorem}

\begin{proof}[Proof of Theorem \ref{A2}]
It follows from \cite[p.\,197]{hsbook} that the collection of
all continuous functions on $K$ is dense on $L^1(K,\mu)$. Then, by a similar argument that used in \cite[Theorem 1.8]{h01}, one can prove Theorem \ref{A2}.
\end{proof}

By Theorem \ref{A2}(i), we immediately have the following corollary.
\begin{corollary}\label{Ac1}
Let $f\in L^1_{\rm loc}(K, \mu)$. Then for almost every $x\in K$, $$|f(x)|\le  M_{\mathcal{D}}^\mu f(x).$$
\end{corollary}

\begin{proof}[Proof of Theorem \ref{lpc}]
Applying Corollary \ref{Ac1} and Theorem \ref{MR1}, we immediately obtain Theorem \ref{lpc}.
\end{proof}

\subsection{Proof of Proposition \ref{MR3}}
Note that
\begin{align}\label{1106}
\begin{split}
\int_K M_{\mathcal{D}}^\mu f(x) d\mu&=2\int_0^\infty
\mu\{x\in K: M_{\mathcal{D}}^\mu f(x)>2\alpha\}d\alpha\\
&\le 2\mu(K)+2\int_1^\infty \mu\{x\in K: M_{\mathcal{D}}^\mu f(x)>2\alpha\}d\alpha.
\end{split}
\end{align}
Decompose $f$ as $f_1+f_2$, where $f_1:=f\chi_{\{x\in K: |f(x)|>\alpha\}}$ and
$f_2:=f-f_1$. Then by the fact that
$M^\mu_{\mathcal{D}}:\ L^\infty(K,\mu)\to L^{\infty}(K,\mu)$, we find that
$M_{\mathcal{D}}^\mu f_2(x)\le\|M_{\mathcal{D}}^\mu f_2\|_{L^{\infty}(K,\mu)}\le \|f_2\|_{L^{\infty}(K,\mu)}\le \alpha,$
which yieds
\begin{align}\label{1107}
\{x\in K: M_{\mathcal{D}}^\mu f(x)>2\alpha\}\subset
\{x\in K: M_{\mathcal{D}}^\mu f_1(x)>\alpha\}.
\end{align}
Therefore, from Corollary \ref{weaki}, we infer that
\begin{align*}
&\int_1^\infty \mu\{x\in K: M_{\mathcal{D}}^\mu f(x)>2\alpha\}d\alpha
 \quad  \quad \text{(by \eqref{1107})}\\
&\quad\le \int_1^\infty \mu\{x\in K: M_{\mathcal{D}}^\mu f_1(x)>\alpha\}d\alpha \quad  \text{(by Corollary \ref{weaki})}\\
&\quad\le \int_1^\infty\frac{4}{\alpha}\int_K |f_1(x)|d\mu d\alpha
\le 4\int_K|f(x)|\int^{\max\{1,|f(x)|\}}_1\frac 1\alpha d\alpha d\mu
=4\int_K|f(x)|\log^+|f(x)|d\mu.
\end{align*}
In the second inequality, we use the fact $f\in L^1(K,\mu)$ due to $|f|\log^+|f|\in L^1(K,\mu)$.
This, together with \eqref{1106}, completes the proof.

\subsection{Proof of Proposition \ref{t16}}
Recall that $\mu$ is a doubling measure if $K$ satisfies the SSC. Therefore, $(K,\mu,d)$ is a space of homogeneous type in the sense of Coifman and Weiss
\cite{cw77}, where $d(x,y):=|x-y|$, with $x,\ y\in K$, denotes the usually Euclidean distance on $\R^d$. We recall the Whitney type covering lemma from \cite[Theorem 3.2]{cw77} as follows.

\begin{lemma}\label{2l2}
Suppose $U\subsetneq K$ is an open bounded set. Then there exists a sequence of balls $\{B_{r_j}(x_j)\}_{j}$ in $K$ with the center $x_j$ and the radius $r_j$ satisfying
\begin{enumerate}[(1)]
\item [\textup{(i)}] $U=\cup_j B_{r_j}(x_j)$;
\item [\textup{(ii)}] there exists a constant $M$ such that $\sum_{j}\chi_{B_{Dr_j}(x_j)}\le M$;
\item [\textup{(iii)}] for any $j$,
$B_{3Dr_j}(x_j)\cap (K\setminus U)\neq \emptyset.$
\end{enumerate}
\end{lemma}

We now prove Proposition \ref{t16}.

\begin{proof}[Proof of Proposition \ref{t16}]
Let $f\in L^1(K,\mu)$ and $M_{\mathcal{D}}^\mu f\in L^1(K,\mu)$.
For any $\alpha\in(0,\infty)$, let $U_{\alpha}:=\{x\in K:\ M_{\mathcal{D}}^\mu f(x)>\alpha\}$. Then, obviously, the set $U_{\alpha}$ is an open bounded set.

We claim that for any $\alpha\in[1,\infty)$,
\begin{align}\label{25e1}
\frac1 \alpha\int_{\{x\in K:\ |f(x)|>\alpha\}}|f(y)|\,d\mu(y)
\le C
\mu(U_{\alpha}).
\end{align}
Since $f\in L^1(K,\mu)$, suppose $\|f\|_{L^1(K,\mu)}>0$.
If $U_{\alpha}=K$, then
\begin{align*}
\alpha \mu(U_{\alpha})&\geq\mu(K)
=\frac{\mu(K)}{\|f\|_{L^1(K,\mu)}}
\int_K|f(y)|d\mu(y)\geq \frac{\mu(K)}{\|f\|_{L^1(K,\mu)}} \int_{\{x\in K:\ |f(x)|>\alpha\}}|f(y)|\,d\mu(y).
\end{align*}
This implies \eqref{25e1}.
If $U_{\alpha}\neq K$, then from Lemma \ref{2l2},
we infer that there exists a sequence of balls
$\{B_{r_j}(x_j)\}_{j}$ satisfying
\begin{enumerate}[(1)]
\item [\textup{(a)}] $U_{\alpha}=\cup_j B_{r_j}(x_j)$;
\item [\textup{(b)}] there exists a constant $M$ such that $\sum_{j}\chi_{B_{Dr_j}(x_j)}\le M$;
\item [\textup{(c)}] for any $j$,
$B_{3Dr_j}(x_j)\cap (K\setminus U_{\alpha})\neq \emptyset.$
\end{enumerate}
By above condition (c), we find that there exists
$y\in B_{3Dr_j}(x_j)$ such that $M_{\mathcal{D}}^\mu f(y)\leq \alpha.$ Therefore, we have
\begin{align}\label{25e2}
\frac1{\mu(B_{3Dr_j}(x_j))}\int_{B_{3Dr_j}(x_j)}|f(y)|d\mu(y)\leq \alpha,
\end{align}
which, together with above conditions (a) and (b), further implies that
\begin{align*}
\mu(U_{\alpha})&=\mu\left(\bigcup_j B_{r_j}(x_j)\right)
\gtrsim\frac1 M\sum_{j}\mu(B_{Dr_j}(x_j)) \gtrsim \sum_{j}
\frac1{\alpha}\int_{B_{3Dr_j}(x_j)}|f(y)|d\mu(y)\quad \text{(by \eqref{25e2})}\nonumber\\
&\gtrsim \frac1{\alpha}\sum_{j}
\int_{B_{r_j}(x_j)}|f(y)|d\mu(y)
\sim \frac1{\alpha}
\int_{\cup_j B_{r_j}(x_j)}|f(y)|d\mu(y)\nonumber\\
&\gtrsim \frac1{\alpha}\int_{\{x\in K:\ |f(x)|>\alpha\}}|f(y)|\,d\mu(y).\quad \text{(by Corollary \ref{Ac1}})
\end{align*}
Here for simplicity, we  use the notation $X \lesssim Y$ (or $X \gtrsim Y$) to denote $X\le C Y$ (or $X\ge C Y$) for some constant $C>0$. We also write $X \sim Y$ if $ X \lesssim Y$ and $X \gtrsim Y$.
Therefore, the inequality \eqref{25e1} holds and we proved the claim.
Thus, we conclude that
\begin{align*}
\int_KM_{\mathcal{D}}^\mu f(x)\,d\mu(x)
&=\int_0^{\infty}\mu(U_{\alpha})\,d\alpha\geq\int_1^{\infty}\mu(U_{\alpha})\,d\alpha
\gtrsim\int_1^{\infty}\frac1{\alpha}\int_{\{x\in K:\ |f(x)|>\alpha\}}|f(y)|\,d\mu(y)\,d\alpha\\
&\gtrsim\int_K|f(y)|\int_1^{\max\{1,|f(x)|\}}
\frac1{\alpha}\,d\alpha\,d\mu(y)
\sim\int_K|f(y)|\log^+|f(y)| d\mu(y).
\end{align*}
Hence $|f|\log^+|f|\in L^1(K,\mu)$ and finishes the proof of
Proposition \ref{t16}.
\end{proof}


\subsection*{Acknowledgements}
Long Huang would like to thank Prof. Liguang Liu for some helpful discussion on
Proposition \ref{uppersc}, and Jinjun Li would like to thank Prof. Hui Rao and Prof. Hua Qiu for some helpful discussion on the properties of p.c.f. self-similar sets.
The project was supported by the National Natural Science Foundations of China (12201139, 12171107, 12271534, 12471119) and the Science and Technology Projects of Guangzhou (SL2024A04J00209).



\medskip

\noindent Long Huang, Jinjun Li and Xiaofeng Wang

\medskip

\noindent School of Mathematics and Information Science,
Guangzhou University, Guangzhou, 510006, The People's Republic of China

\smallskip

\noindent {\it E-mails}: \texttt{longhuang@gzhu.edu.cn}

\noindent\phantom{{\it E-mails:}}
\texttt{lijinjun@gzhu.edu.cn}

\noindent\phantom{{\it E-mails:}}
\texttt{wxf@gzhu.edu.cn}

\end{document}